\documentclass[review]{elsarticle}
\usepackage{lineno,multirow,rotating,amsmath,setspace}
\usepackage{color}
\usepackage{amssymb, bm}
\usepackage{caption}
\usepackage{multicol}
\usepackage{geometry}
\usepackage{todonotes}
\usepackage{upgreek}
\modulolinenumbers[5]
\usepackage{nomencl}
\usepackage{graphicx}
\usepackage{hyperref}

\usepackage{lmodern}

\usepackage{url}
\usepackage{amssymb,amsthm,amsfonts,amstext}
\usepackage{amsmath}

\usepackage{graphicx}
\usepackage{enumerate}
\numberwithin{equation}{section}
\usepackage{a4wide}
\usepackage{todonotes}
\usepackage[mathscr]{eucal}

\makeatletter
\newcommand*\bigcdot{\mathpalette\bigcdot@{.7}}
\newcommand*\bigcdot@[2]{\mathbin{\vcenter{\hbox{\scalebox{#2}{$\m@th#1\bullet$}}}}}
\makeatother

\usepackage{tikz}
\usetikzlibrary{backgrounds}
\usetikzlibrary{patterns,fadings}
\usetikzlibrary{arrows,decorations.pathmorphing}
\usetikzlibrary{calc}
\definecolor{light-gray}{gray}{0.95}
\usepackage{graphicx}
\usepackage{epsfig}
\usetikzlibrary{shapes}
\usetikzlibrary{decorations}
\usetikzlibrary{decorations.pathmorphing}
\usetikzlibrary{decorations.pathreplacing}
\usetikzlibrary{decorations.shapes}
\usetikzlibrary{decorations.text}
\usetikzlibrary{decorations.markings}
\usetikzlibrary{decorations.fractals}
\usetikzlibrary{decorations.footprints}

\renewcommand{\epsilon}{\varepsilon}

\newcommand{\mf}[1]{{\mathfrak #1}}
\newcommand{\mb}[1]{{\mathbf #1}}
\newcommand{\bb}[1]{{\mathbb #1}}

\newtheorem{lemma}{Lemma}
\newtheorem{remark}{Remark}

\numberwithin{theorem}{section} % important bit
\numberwithin{lemma}{section} % important bit 

\renewcommand{\leq}{\leqslant}
\renewcommand{\geq}{\geqslant}
\renewcommand{\le}{\leqslant}
\renewcommand{\ge}{\geqslant}

\journal{arXiv\ Mathematics}
\begin{document}
%%%%%%%%%%%%%%%%%%%%%%%%%%%%%%%%%%%%%%
	\begin{frontmatter}
	\title{Generalized integral transform method for solving multilayer diffusion problems}
	%%%%%%%%%%%%%%%%%%%%%%%%%%%%%%%5
	\author{Mohamed Akel$^{a,b,*}$}
	\ead{makel@sci.svu.edu.eg , makel65@yahoo.com}
	\author{Hillal M. Elshehabey$^{a,b}$}
	\ead{hilal.hilal@sci.svu.edu.eg}
	\author{Ragaa Ahmed$^{a,b}$}
	\ead{ragaa\_ahmed@sci.svu.edu.eg}
	\cortext[cor1]{Corresponding author.}
	
	\address{ $^a$Mathematics Department, Faculty of Science, South Valley University, Qena 83523, Egypt\\
		$^b$Academy of Scientific Research and Technology (ASRT), 101 Kasr Al-Ainy St., Cairo 11516, Egypt}
	
	%%%%%%%%%%%%%%%%%%%%%%%%%%%%%%%%%%%%%%%%%%%%%%%%%%%%%%%%%%%	
		\begin{abstract}
Multilayer diﬀusion problems have found significant important that they arise in many medical, environmental and industrial applications of heat and mass transfer. In this article, we study the solvability of one-dimensional nonhomogeneous multilayer diffusion problem. We use a new generalized integral transform, namely, $ \bb M_{\rho,m} -$transform %[Acta Mathematica Scientia 35, 6 (2015), 1386–1400]. 
[Srivastava et al., https://doi.org/10.1016/S0252-9602(15)30061-8]. 
First, we reduce the nonhomogeneous multilayer diffusion problem into a sequence of one-layer diffusion problems including time-varying given functions, followed by solving a general nonhomogeneous one-layer diffusion problem via the $ \bb M_{\rho,m} -$transform. %[Srivastava et al., https://doi.org/10.1016/S0252-9602(15)30061-8], 
Hence, by means of general interface conditions, a renewal equations' system is determined. Finally, the $ \bb M_{\rho,m} -$transform and its analytic inverse are used to obtain an explicit solution to the renewal equations' system. Our results generalize those ones in %[Journal of Mathematical Analysis and Applications 444,1(2016),475–502].
[Rodrigo and Worthy, http://dx.doi.org/10.1016/ j.jmaa.2016.06.042].
		\end{abstract}
		
		\begin{keyword}
			\texttt{Multilayer, Diffusion equation, Integral transform, Boundary value problem.}
		\end{keyword}
	\end{frontmatter}
%	\linenumbers
%
{\bf Math. Classifications:} 35K05, 44A05, 44A35, 58J35, 76R50, 76S05	
%%%%%%%%%%%%%%%%%%%%%%%%%%%%%%%%%%%%%%%%%%%%%%%%%%%%%%%%%%%%%%%
\section{Introduction} 
%%%
The multilayer diffusion problems are typical models for variety of solute transport phenomena in layered permeable media, such as advection, dispersion and reaction diffusions (\cite{goltz2017analytical,hou2021boundary,leij1991mathematical,liu1998analytical,liu2019solvability,van1982analytical}). These problems have had their importance due to their natural prevalence in a remarkable large number of applications such as chamber-based gas fluxes measurements \cite{liu2009multi}, contamination and decontamination in permeable media  \cite{liu1998analytical,liu2008analytical}, drug eluting stent \cite{mcginty2011modelling,pontrelli2007mass},  drug absorption \cite{addicks1989mathematical,simon2005analytical}, moisture propagation in woven fabric composites \cite{pasupuleti2011modelling},  permeability of the skin \cite{mitragotri2011mathematical}, and  wool-washing\cite{caunce2008spatially}.
 
%\textcolor{red}{
As epidemiological models, reaction-diffusion problems are widely used to model and analyze the spread of diseases such as the global COVID-19 pandemic caused by resulted from SARS-CoV2.  These models describe the spatiotemporal prevalence of the viral pandemic, and apprehend the dynamics depend on human habits and geographical features. The models estimate a qualitative harmony between the simulated prediction  of the local spatiotemporal spread of a pandemic and the epidemiological collected datum. See \cite{raimundez2021covid,viguerie2021simulating}. These data-driven emulations can essentially inform the respective authorities to purpose efficient pandemic-arresting measures and foresee the geographical distribution of vital medical resources. Moreover, such studies explore alternate scenarios for the repose of lock-down restrictions based on the local inhabitance  densities and the qualitative dynamics of the infection. For more applications one can refer e. g., to \cite{du2018partial,hickson2009critical}.

Although the numerical methods are usually applied to solve the diffusion problems, especially in the heterogeneous permeable media, the analytic solutions when available, are characterized by their exactness and continuity in space and time. In this work, we focus on analytic solutions of certain nonhomgeneous diffusion problems in multilayer permeable media. Here, the retardation factors are assumed to be constant, the dispersion coefficients vary across layers, but being constants within each layer, and the free terms are (arbitrary) time-varying functions.

Analytic and semi analytic solutions of multilayer diffusion problems are developed by using the Laplace integral transform \cite{carr2020new}, \cite{carr2016semi}, \cite{carr2018semi}, \cite{carr2018modelling},  \cite{de2002analytic},\cite{guerrero2013analytical}, \cite{hickson2009critical}, \cite{liu1998analytical}, 
\cite{rodrigo2016solution}, \cite{sun2004transient},
\cite{zimmerman2016analytical}. Applying Laplace transforms, to solve multilayer diffusion problems, has advantages as an applicable tool in  handling different types of boundary conditions and averts solving complicated transcendental equations as in demand by eigenfunction expansion methods. Further works involving the Laplace transform have studied permeable layered reaction diffusion problem in \cite{chen2016assessing}, \cite{park2009one}. Solutions obtained in these works are restricted to two layers as well as obtaining the inverse Laplace transform numerically. 

In the current work, we aim to extend, generalize and merge results in \cite{carr2016semi}, \cite{carr2018semi}, \cite{park2009one}, \cite{rodrigo2016solution} and \cite{zimmerman2016analytical} to solve certain nonhomgeneous diffusion problems in one-dimensional n-lyared media. We use a new generalized integral transform recently introduced in \cite{sri015}. The obtained solutions are applicable to more general linear nonhomogeneous diffusion equations, finite media consisting of arbitrary many layers, continuity and dispersive flow at the contact interfaces between sequal layers and  transitory boundary conditions of arbitrary type at the inlet and outlet. To the best knowledge of the authors, analytical solutions verifying all the above mentioned conditions have not previously reported in literature which strongly motivates this current work.

In the remaining part of this introductory section, in Subsection 1.1 the multilayer diﬀusion problem is described and then it is reformulated as a sequence of one-layer diﬀusion problems having boundary conditions including given time-depending functions. Basic properties for $\bb M_{\rho,m}-$transform that will be needed in this work are stated in Subsection 1.2. The remaining sections are constructed as follows: Section 2 is devoted to , we discuss the solvability of a general linear nonhomogeneous one-layer diffusion problem with arbitrary time-varying data, using the $\bb M_{\rho,m}-$transform. Section 3 is devoted to our main multilayer diffusion problem, where in Subsection 3.1 we solve a two-layer problem to shed light on the basic idea by considering this simple case. Further, in Subsection 3.2, we return to benfit from the results obtained in Section 2 and Subsection 3.1 to solve the main multilayer diffusion problem \eqref{gPDE}-\eqref{Inner Bc2}, see Subsection \ref{MM1.1} below. 
%
% 
%%%%%%%%%%%%%%%%%%%%%%%%%%%%%%%%%%%%%%%%%%%%%%%%%%%%%%%%%%%%%%%%%%%%%%%%%%%%%%%%%%%%%%%%%%%%%%%%%%%%%%
\subsection{\textbf{Mathematical modeling for nonhomogeneous n-layer diffusion systems}} \label{MM1.1} 
A one-dimensional diffusion problem in an n-layered permeable medium is set out as follows. Let 
\[\alpha=x_{0}<x_{1}<\cdots<x_{n-1}<x_{n}=\beta \]
be a finite partition of the interval $[\alpha,\beta]$. In each subinterval $[x_{j-1},x_{j}]$, with $j=1,2,\cdots,n$, the component function $\varphi_{j}(x,t)$ satisfies the partial differential equation (PDE)
\begin{equation}\label{gPDE}
\frac{\partial \varphi_{j}}{\partial t}=d_{j}\frac{\partial^{2} \varphi_{j}}{\partial x^{2}}+ \lambda(t,\uptau)r_{j}(x,t),\   x\in(x_{j-1},x_{j}),\ \ \ t,\uptau>0,
\end{equation}
where $d_{j}\geq 0$, for all $1\leq j\leq n$, are the diffusion coefficients and 
\begin{equation}\label{2.5}
\lambda(t,\uptau)=\bigg(\frac{t^{m}}{\uptau^{m}}+\uptau^{m}\bigg)^{-\rho},\quad t \geq 0,\; \uptau> 0,%\; m\in  \bb Z_{+},\rho\in  \bb C,Re\rho >0. \uptau ,t \ge 0,m\in  \bb R,\rho\in  \bb C,Re\rho >0.
\end{equation}
with $ m\in  \bb Z_{+} =\{ 1,2,3,\cdots\},\ \rho\in  \bb C, \mbox{Re}(\rho) >0$. Here, the function-term $\lambda(t,\uptau)r_{j}(x,t)$ physically means the external source term that could be applied to the diffusion equation with $r_{j}(x,t)$  depends on time and space while the other factor of the source term i.e., $p$ depends only on time. This last term could be for instance, a periodic-time magnetic source.\\ 
The initial conditions (ICs) are assumed as
\begin{equation}\label{ICs}
\varphi_{j}(x,0)=\eta_{j}(x), \quad x\in [x_{j-1},x_{j}],\, 1\leq j\leq n.
\end{equation}	
The boundary conditions (BCs) are posited as
\begin{itemize}
	\item The outer BCs (at the inlet $x=\alpha$ and the outlet $x=\beta$) are general Robin boundary conditions as
\begin{align}%\iota \ell \l \imath instes f alpha beta gamma and delta
& \imath \varphi_{1}(\alpha,t)+\iota \frac{\partial \varphi_{1}}{\partial x}(\alpha,t)=\lambda(t,\uptau)\zeta(t),\label{inlet BCs}\\
& \ell \varphi_{n}(\beta,t)+\l \frac{\partial \varphi_{n}}{\partial x}(\beta,t)=\lambda(t,\uptau)\xi(t),\label{outlet BCs}
\end{align}
for all $t \ge 0$, with $\imath, \iota, \ell$ and $\l$ are constants satisfying, $ |\imath|+|\iota|>0,\ |\ell|+|\l|>0 $.
\item The inner BCs (the interface conditions) are
\begin{align}
  \varphi_{j}(x_{j},t)&=\Lambda_{j} \varphi_{j+1}(x_{j},t),%\quad t \ge 0,\, j=1,2,\cdots,n-1,
  \label{Inner Bc1}\\
  \nu_j \varphi_{j}(x_{j},t) +\mu _j \frac{\partial \varphi_{j}}{\partial x}(x_{j},t)&=\nu_{j+1} \varphi_{j+1}(x_{j},t) +\mu_{j+1} \frac{\partial \varphi_{j+1}}{\partial x}(x_{j},t)\label{Inner Bc2}
\end{align}
for all $t \ge 0$, with $ |\nu_j |+|\mu_j |>0$ for all $ j=1,2,\cdots,n-1.$
\end{itemize}
%%%%
For appropriate given functions $\eta_{1},\cdots,\eta_{n},\zeta$ and $\xi$, we are going to find an analytic solution of the problem \eqref{gPDE}-\eqref{Inner Bc2} using the $\bb M_{\rho,m}-$generalized integral transform, introduced recently in \cite{sri015}. Problem \eqref{gPDE}-\eqref{Inner Bc2} can be reduced into the following sequence of one-layer diffusion problems.
\begin{itemize}
	\item In the inlet layer i.e., $x \in [x_{0},x_{1}]$
	\begin{equation}\label{BPV in 1 layer}
	\begin{split}
	& \frac{\partial \varphi_{1}}{\partial t}=d_{1} \frac{\partial^{2} \varphi_{1}}{\partial x^{2}} + \lambda(t,\uptau)r_{1}(x,t),\ \ x\in(x_{0},x_{1}),\ \,t,\uptau >0,\\
	& \varphi_{1}(x,0)=\eta_{1}(x),\quad x\in [x_{0},x_{1}],\\
	& \imath \varphi_{1}(x_{0},t) +\iota \frac{\partial \varphi_{1}}{\partial x}(x_{0},t)=\lambda(t,\uptau)\zeta_{1}(t),\quad t \ge 0,\ \uptau>0,\\
	&\nu_{1} \varphi_{1}(x_{1},t) +\mu_{1} \frac{\partial \varphi_{1}}{\partial x}(x_{1},t)=\lambda(t,\uptau)\xi_{1}(t),\quad t \ge 0,\ \uptau>0.
	\end{split}    
    \end{equation}
    \item In the interior layers i.e.,  $x\in [x_{j-1},x_{j}]$, $2\le j\le n-1$
    \begin{equation}\label{BPV in j layer}
    \begin{split}
    & \frac{\partial \varphi_{j}}{\partial t}=d_{j} \frac{\partial^{2} \varphi_{j}}{\partial x^{2}} + \lambda(t,\uptau)r_{j}(x,t),\quad x\in(x_{j-1},x_{j}),\ \ t,\uptau>0,\\
    & \varphi_{j}(x,0)=\eta_{j}(x),\quad x\in [x_{j-1},x_{j}],\\
    & \nu_{j} \varphi_{j}(x_{j-1},t) +\mu_{j} \frac{\partial \varphi_{j}}{\partial x}(x_{j-1},t)=\lambda(t,\uptau)\zeta_{j}(t),\quad t \ge 0,\ \uptau>0,\\
    &\nu_{j} \varphi_{j}(x_{j},t) +\mu_{j} \frac{\partial \varphi_{j}}{\partial x}(x_{j},t)=\lambda(t,\uptau)\xi_{j}(t),\quad t \ge 0,\ \uptau>0.
    \end{split}    
    \end{equation}
    \item In the outlet layer i.e.,  $x\in [x_{n-1},x_{n}]$
    \begin{equation}\label{BPV in n layer}
    \begin{split}
    & \frac{\partial \varphi_{n}}{\partial t}=d_{n} \frac{\partial^{2} \varphi_{n}}{\partial x^{2}} + \lambda(t,\uptau)r_{n}(x,t),\quad x\in(x_{n-1},x_{n}),\ \ t,\uptau>0\\
    & \varphi_{n}(x,0)=\eta_{n}(x),\quad x\in [x_{n-1},x_{n}],\\
    & \nu_{n} \varphi_{n}(x_{n-1},t) +\mu_{n} \frac{\partial \varphi_{n}}{\partial x}(x_{n-1},t)=\lambda(t,\uptau)\zeta_{n}(t),\quad t \ge 0,\ \uptau>0,\\
    &\ell \varphi_{n}(x_{n},t) +\l \frac{\partial \varphi_{n}}{\partial x}(x_{n},t)=\lambda(t,\uptau)\xi_{n}(t),\quad t \ge 0,\ \uptau>0.
    \end{split}    
    \end{equation}
\end{itemize}
\begin{remark}\label{rem1}
Each of the initial boundary value problems \eqref{BPV in 1 layer}-\eqref{BPV in n layer} is a case of one-layer nonhomogeneous diffusion problem that will be discussed in Section 2 below.
\end{remark}
Now, in view of the inner boundary conditions \eqref{Inner Bc1}-\eqref{Inner Bc2}, the time-varying functions $\zeta_{j}$ and $\xi_{j}$ for all $2\leq j\leq n$ are subject to
\begin{equation}\label{g_j=h_j-1}
\zeta_{j}(t)=\xi_{j-1}(t),\quad 2\le j\le n,
\end{equation}
%%%%%%
so that
\begin{equation}\label{zeta_j}
	\zeta_{j}(t)=\begin{cases}
		\begin{split}
	\zeta(t),\quad\quad\quad \quad& \quad j=1,	\\
\xi_{j-1}(t),\quad\quad\quad \quad& \quad 2\le j\le n,
\end{split}
	\end{cases}
\end{equation}
and
\begin{equation}\label{xi_{j}}
	\xi_{j}(t)=\begin{cases}
		\begin{split}
		& \left(\nu_j \varphi_{j}(x_{j},t) +\mu_{j} \frac{\partial \varphi_{j}}{\partial x}(x_{j},t)\right)\lambda^{-1}(t,\uptau), & \quad &1\le j\le n-1,	\\
		& \xi(t), & \quad &j=n.
		\end{split}
	\end{cases}
\end{equation}

While, the outer boundary data $\zeta_{1}(t)=\zeta(t)$ and $\xi_{n}(t)=\xi(t)$ are given in \eqref{inlet BCs} and \eqref{outlet BCs}, respectively, the functions $\zeta_{j}\,(2\le j\le n)$ can be determined once we specify the functions $\xi_{j}\, (1\le j\le n-1)$. Hence, we have to find $\xi_{j},\ 1 \le j \le n-1$. To do so, we should use the first matching condition \eqref{Inner Bc1}.
%

%%%%%%%%%%%%%%%%%%%%%%%%%%%%%%%%%%%%%%%%%%%%%%%%%%%%%%%%%%%%%%%%%%%%%%%%%%%%%%%%%%%%%%%%%%%%%%%%%%%%%%
\subsection{\textbf{Srivastava-Luo-Raina generalized integral transform}}%$ \bb M_{\rho,m} -$transform}
\label{subsec1.2}
%\todo[color=green]{$\psi$  is used down }
In \cite{sri015}, Srivastava et al. introduced the following generalized integral transform
\begin{align}\label{1.1}
	\bb{M}_{\rho,m}[\varphi(t)](s,\uptau )=\int_{0}^{\infty}\frac{e^{-st} \varphi(\uptau t)}{(t^{m}+\uptau^{m})^{\rho}}dt,
\end{align}
for a continuous (or piecewise continuous) function $\varphi$ on $[0,\infty)$, where $\rho\in\bb{C}; \operatorname{Re}(\rho)\geq 0; m\in\bb{Z}_{+}$, $s>0$ is the transform variable and $\uptau>0$ is a parameter. The basic properties of the $ \bb M_{\rho,m} -$transform are given in \cite{sri015}. Next we recall some of these properties, which are needed in the present work. Indeed, as introduced in \cite{sri015} the $ \bb M_{\rho,m} -$transform is closely related with the well-known integral transforms, the Laplace, natural and Sumudu transforms. The Laplace transform is defined by
\begin{align}\label{1.2}
	{\bb L} [\varphi(t)](s)=\int_{0}^{\infty}e^{-st}\varphi(t) dt,\; \operatorname{Re}(s)>0.
\end{align}
So, from (\ref{1.1}) and (\ref{1.2}) we have the following duality relations
\[\bb{L}[\varphi(t)](s)=\bb{M}_{0,m}[\varphi(t)](s,1),\; \operatorname{Re}(s)>0,\]
\[\bb{M}_{\rho,m}[\varphi(t)](s,\uptau)=\bb{L}\left[\frac{\varphi(\uptau t)}{(t^{m}+\uptau^{m})^{\rho}}\right](s),\; s,\uptau>0,\]
\[\bb{M}_{\rho,m}[\varphi(t)](s,\uptau)=\frac{1}{\uptau}\bb{L}\left[\frac{\varphi(t)}{\left( \frac{t^{m}}{\uptau^{m}}+\uptau^{m}\right)^{\rho}}\right](\frac s \uptau),\; s,\uptau>0,\]
and,
\[\bb{M}_{\rho,m}\left[\left( \frac{t^{m}}{\uptau^{m}}+\uptau^{m}\right)^{\rho}\varphi(t)\right](s,\uptau)= \bb{L}\left[\varphi(\uptau t)\right](s),\; s,\uptau>0.\]
Setting $\rho=0$ in (\ref{1.1}), we recover the natural transform defined as (see \cite{BS}, \cite{SB})
\begin{align}\label{1.3}
	\bb{N}[\varphi(t)](s,\uptau)=\int_{0}^{\infty} e^{-s t} \varphi(\uptau t) dt,\; s>0, \uptau>0. 
\end{align}
Thus, we have the following $\bb{M}_{\rho,m}-\bb{N}-$transforms duality
$$
\bb{N}[\varphi(t)](s,\uptau)=\bb{M}_{0,m}[\varphi(t)](s,\uptau),
$$
\begin{align}\label{1.4}
	\bb{M}_{\rho,m}[\varphi(t)](s,\uptau)=\bb{N}\left[\frac{\varphi(t)}{\left(\frac{t^{m}}{\uptau^{m}}+\uptau^{m}\right)^{\rho}}\right](s,\uptau),\; s>0, \uptau>0,
\end{align}
and,
\begin{align}\label{1.5}
	\bb{M}_{\rho,m}\left[\left(\frac{t^{m}}{\uptau^{m}}+\uptau^{m}\right)^{\rho}\varphi(t)\right](s,\uptau)=
	\bb{N}\left[\varphi(t)\right](s,\uptau),\; s>0, \uptau>0.
\end{align}
The Sumudu transform is defined by (\cite{BK},\cite{BKK} and \cite{KB})
\begin{align*}
	\bb{S}[\varphi(t)](\uptau)=\int_{0}^{\infty}e^{-t} \varphi(\uptau t) dt,\; \uptau>0
\end{align*}
Thus,
\[\bb{S}[\varphi(t)](\uptau)=\bb{M}_{0,m}[\varphi(t)](0,\uptau),\;\uptau>0,\]
and, 
%%%%%%%%%%%%%%%%%%%%%%%
\begin{align*}%\label{1.6}
	\bb{M}_{\rho, m}[\varphi(t)](s, \uptau)=
	\frac{1}{s} \bb{S}\left[\frac{\varphi(t)}{\left(\frac{t^{m}}{\uptau^{m}}+\uptau^{m}\right)^{\rho}} \right] \left(\frac{\uptau}{s}\right),\; s,\uptau>0.
\end{align*}

%%%%%%%%%%%%%%%%%%%%%%%%%%%%%
%
% On basis of the above mentioned
Based on these dualities of the $\bb M_{\rho,m}-$transform (\ref{1.1}) and these well-known integral transforms it seems to be interesting to apply the $\bb M_{\rho,m}-$transform (\ref{1.1}) in solving a variety of boundary and initial-boundary problems. In this context, we recall the following results \cite{sri015}:
%\begin{lemma}\label{lem1.1}
\begin{itemize}
	\item 	Let $\varphi^{(n)}(t)$ be the $n^{th}-$order $t-$derivative of the function $\varphi(t)$ and $|\varphi(t)|\leq K e^{t/\gamma}$ with $K>0, \gamma>0$. Then,
	\begin{align}\label{1.7}
		\bb{M}_{\rho, m}\left[\left(\frac{t^{m}}{\uptau^{m}}+\uptau^{m}\right)^{\rho} \varphi^{(n)}(t)\right](s,\uptau)
		=\frac{s^{n}}{\uptau^{n}}\bb{N}[\varphi(t)](s,\uptau)- \sum_{k=0}^{n-1}\frac{s^{k}}{\uptau^{k+1}} \varphi^{(n-k-1)}(0)
	\end{align}
	where $\bb{N}[\varphi(t)](s,\uptau)$ is defined by (\ref{1.3}). 
%\end{lemma}
Using the duality \eqref{1.5} in \eqref{1.7}, we find
\begin{align}\label{N for derivatives}
	\bb{N} \left[ \varphi^{(n)}(t)\right](s,\uptau)
	=\frac{s^{n}}{\uptau^{n}}\bb{N}[\varphi(t)](s,\uptau)- \sum_{k=0}^{n-1}\frac{s^{k}}{\uptau^{k+1}} \varphi^{(n-k-1)}(0), \quad n=0,1,\cdots. 
\end{align}
%%%%%%%%%%
\item Again, using the dualities stated before a convolution formula for the $\bb M_{\rho,m} -$transform \eqref{1.1} can be obtained as follows. Here, the convolution for the Laplace transform will be considered, that is, for the functions $\varphi$ and $\psi$, the convolution formula is given as
\[(\varphi*\psi)(t)=\int_{0}^{t}\varphi(x)\psi(t-x)dx=\int_{0}^{t}\varphi(t-x)\psi(x)dx.\]
If 
$\Phi(s,\uptau)=\bb M_{\rho,m}[\varphi(t)](s,\uptau)$ and $\Psi(s,\uptau)=\bb M_{\rho,m}[\psi(t)](s,\uptau)$, then
\begin{align*}
	\uptau \Phi(s,\uptau) \Psi(s,\uptau)&=\uptau \int_{0}^{\infty}\frac{e^{-st_{1}} \varphi(\uptau t_{1})}{\left(t_{1}^{m}+\uptau^{m}\right)^{\rho}}dt_{1}
	\int_{0}^{\infty}\frac{e^{-st_{2}} \psi(\uptau t_{2})}{\left(t_{2}^{m}+\uptau^{m}\right)^{\rho}}dt_{2}\\
	&=\uptau \int_{0}^{\infty} \int_{0}^{\infty} e^{-s(t_{1}+t_{2})}\tilde{\varphi}(\uptau t_{1})\tilde{\psi}(\uptau t_{2})dt_{1} dt_{2},
\end{align*}
where
\[\tilde{\varphi}(t)=\frac{\varphi(t)}{\left(\frac{t^{m}}{\uptau^{m}}+\uptau^{m}\right)^{\rho}},\ \tilde{\psi}(t)=\frac{\psi(t)}{\left(\frac{t^{m}}{\uptau^{m}}+\uptau^{m}\right)^{\rho}}.\]
Setting $t_{1}+t_{2}=t$ in the last equality, one gets
\begin{align*}
	\uptau \Phi(s,\uptau) \Psi(s,\uptau)&=
	\uptau \int_{0}^{\infty} \int_{t_{2}}^{\infty} e^{-st}\tilde{\varphi}(\uptau(t-t_{2})) \tilde{\psi}(\uptau t_{2})dt dt_{2}\\
	&=\uptau \int_{0}^{\infty}e^{-st} dt \int_{0}^{t} \tilde{\varphi}(\uptau(t-t_{2})) \tilde{\psi}(\uptau t_{2}) dt_{2},
\end{align*}
%%%%%%
here, changing of the integral order is used. Thus, using the duality of the $\bb M_{\rho,m}$ and $\bb N$ transforms (see \eqref{1.4}), we find
\begin{equation}\label{M-Convolution}
	\uptau \Phi(s,\uptau) \Psi(s,\uptau)=\bb N[\tilde{\varphi}*\tilde{\psi}](s,\uptau).
\end{equation}
%%%%%%
\begin{remark}\label{rem2}
	If we put $\rho=0$ in \eqref{M-Convolution}, the case being interesting later in our work, then we get
	\begin{equation}\label{N-Convolution}
		\uptau \bb N[\varphi(t)](s,\uptau) \bb N[\psi(t)](s,\uptau)=\bb N[(\varphi*\psi)(t)](s,\uptau).
	\end{equation}	
\end{remark}
%%%%%%%%%%%%%%%%%%%%%%%%%%%%%%%%%%%%%%%%%%%%%%%%%%%%%%%%%%%%%%%%%%%%%%%%%%%%%%%%%%%%%%%%%%%%%%%%%%%%%%%
\item Once again, using the dualities stated before an inversion formula of the $\bb \bb M_{\rho,m} -$transform \eqref{1.1} is given (see, \cite[Theorem 4.1]{sri015}) as
\begin{align}\label{M-Inversion}
	\varphi(t)&=\left(\frac{t^m}{\uptau^m}+\uptau^{m}\right)^{\rho} \bb L^{-1} \left\{\bb M_{\rho,m} [\varphi(t)](s,\uptau)\right\}\left(\frac{t}{\uptau}\right) \nonumber \\
	&=\frac{1}{2\pi i}\left(\frac{t^{m}}{\uptau^{m}}+\uptau^{m}\right)^{\rho} \int_{c-i\infty}^{c+i\infty} e^{\frac{st}{\uptau}}
	\bb M_{\rho,m}[\varphi(t)](s,\uptau) ds,\quad c,\uptau>0,
\end{align}
as long as the integral converges absolutely. In case, when $\rho=0$ one obtains the following inversion formula of the natural transform (\cite[Theorem 5.3]{BS})
\begin{align}\label{N-Inversion}
	\varphi(t)=\frac{1}{2\pi i} \int_{c-i\infty}^{c+i\infty} e^{\frac{st}{\uptau}}
	\bb N[\varphi(t)](s,\uptau) ds,\quad c,\uptau>0.
\end{align}
The residue theorem (see e.g. \cite{lang2013complex}) is usually used to calculate the contour integrals in \eqref{M-Inversion} and \eqref{N-Inversion}. 
\end{itemize}
%%%%%%%%%%%%%%%%%%%%%%%%%%%%%%%%%%%%%%%%%%%%%%%%%%%%%%%%%%%%%%%%%%%%%%%%%%%
%%%%%%%%%%%%%%%%%%%%%%%%%%%%%%%%%%%%%%%%%%%%%%%%%%%%%%%%%%%%%%%%%%%%%%%%%%%%%%%%%%%%%%%%%%%%%%%%%%%%%
\section{One-layer nonhomogeneous diffusion system}
 \label{sec2}
 Now, we investigate the solvability for the following one-layer nonhomogeneous initial-boundary value problem 
\begin{eqnarray}%\label{IBVP}
 	&&\frac{\partial \varphi}{\partial t}=d \frac{\partial^{2} \varphi}{\partial x^{2}}+\lambda(t,\uptau)r(x ,t), \quad x \in (\alpha,\beta),\ \ t,\uptau >0,\qquad \qquad \qquad \qquad  \label{2.1} %\;\lambda(t,\uptau)\neq 0 
 	 \\
    &&\varphi(x,0)=\eta(x),\quad x \in [\alpha,\beta],\label{2.2} 
    \\
 	&&\imath \varphi(\alpha,t)+\iota \frac{\partial \varphi}{\partial x}(\alpha,t)=\lambda(t,\uptau)\zeta(t), \quad t \geq 0,\; \uptau> 0, \label{2.3} % t \ge 0,\uptau \in  \bb R 
 	\\
 	&&\ell \varphi(\beta,t)+\l \frac{\partial \varphi}{\partial x}(\beta,t)=\lambda(t,\uptau)\xi(t), \quad t \geq 0,\; \uptau> 0, \label{2.4} % t \ge 0,\uptau \in  \bb R, 
  \end{eqnarray}
 where $ d, \imath, \iota ,\ell$ and $\l$ are constants such that $ |\imath|+|\iota|>0,\ |\ell|+|\l|>0 $, and $ p,r,\eta,\zeta $ and $ \xi $ are given functions with $p$ as in \eqref{2.5}.
 
Applying the $ \bb M_{\rho,m} -$transform defined by (\ref{1.1}) to (\ref{2.1}), yields
\begin{equation}\label{2.6}
 	\bb M_{\rho,m}\left[\frac{1}{\lambda(t,\uptau)}\varphi_{t}(x,t)\right]=d\ \bb M_{\rho,m}\left[\frac{1}{\lambda(t,\uptau)}\varphi_{xx}(x,t) \right]+\bb M_{\rho,m}[r(x,t)].
\end{equation}
Using the duality of the $ \bb M_{\rho,m} -$transform and the natural transform given by (\ref{1.5}) and \eqref{1.7}, Eq.(\ref{2.6}) can be reduced to   
\begin{equation}\label{2.7}
 	\dfrac{s}{\uptau d}\bb N[\varphi(x,t)](x;s,\uptau)-\bb N[\varphi_{xx}(x,t)](x;s,\uptau)-\dfrac{1}{\uptau d }\varphi(x,0)= \dfrac{1}{d}\bb N [\lambda(t,\uptau)r(x,t)](x;s,\uptau).
 \end{equation}
 where $\bb{N}[\varphi(t)](s,\uptau)$ is defined by (\ref{1.3}). Setting 
 \begin{equation}\label{2.8}
 	\hat{\varphi}(x;s,\uptau)=\bb N[\varphi(x,t)](x;s,\uptau),
 \end{equation}
 then, (\ref{2.7}) can be expressed as %eq.(1.10) can be written as 
 \begin{equation}\label{2.9}
 	 		\hat{\varphi}_{xx}(x;s,\uptau)-\dfrac{s}{\uptau d}\hat{\varphi}(x;s,\uptau)=F(x;s,\uptau), 
\end{equation}
where
\begin{equation} \label{F}	 		
 F(x;s,\uptau)=-\dfrac{1}{d}\bb N [\lambda(t,\uptau)r(x,t)](x;s,\uptau)-\dfrac{1}{\uptau d }\eta(x).
  \end{equation}
Applying the variation of parameters method to the nonhomogeneous equation \eqref{2.9}, gives the general solution as
 \begin{equation}\label{2.10} %\label{F} 
 	\hat{\varphi}(x;s,\uptau )=A \cosh\sqrt{\frac{s}{\uptau  d}}x+B \sinh\sqrt{\frac{s}{\uptau  d}}x+\sqrt{\frac{\uptau  d}{s}}\int_{\alpha}^{x}F(y;s,\uptau)\sinh\sqrt{\frac{s}{\uptau  d}}(x-y)dy,
 \end{equation}
 where $ A $ and $ B $ are arbitrary invariants which can depend on $ s $ and $ \uptau $.\\
 Differentiating \eqref{2.10} with respect to $ x $, gives % then we get 
 \begin{equation}\label{2.11} %\label{U}
 %\resizebox{1.05\hsize}{!}{$
\begin{split}
  		\hat{\varphi}_{x}(x;s,\uptau)=&A\sqrt{\frac{s}{\uptau d}} \sinh\sqrt{\frac{s}{\uptau d}}x+B\sqrt{\frac{s}{\uptau d}} \cosh\sqrt{\frac{s}{\uptau d}}x\\
  		&+ \int_{\alpha}^{x}F(y;s,\uptau)\cosh\sqrt{\frac{s}{\uptau d}}(x-y)dy.
\end{split}
 \end{equation}
Transforming the boundary conditions \eqref{2.3} and \eqref{2.4}, implies
\begin{equation}\label{2.12}  %\label{Ut}
 	\begin{split}
 		&\imath \hat{\varphi}(\alpha;s,\uptau)+\iota \hat{\varphi}_{x}(\alpha;s,\uptau)=\bb M_{\rho, m}[\zeta(t)]=\bb N [\lambda(t,\uptau)\zeta(t)]\\
 		&\ell \hat{\varphi}(\beta,,s,t)+\l \hat{\varphi}_{x}(\beta;s,t)=\bb M_{\rho, m}[\xi(t)]=\bb N [\lambda(t,\uptau)\xi(t)].
 	\end{split}
\end{equation}
 For simplicity, we set the following vector notations
% $  \mf a=(\imath,\iota) $ and $  \mf b=(\ell,\l) $. Also, we define the following vectors
 %
 \begin{equation}\label{2.13}  %\label{HH}
 	\begin{split}
 		\mf a=(\imath,\iota),& \ \mf b=(\ell,\l),  \\
 		\mf L(y;s,\uptau)&=\bigg(\cosh\sqrt{\frac{s}{\uptau d}}y,\sqrt{\frac{s}{\uptau d}}\sinh\sqrt{\frac{s}{\uptau d}}y\bigg),\\
 		\mf C(y;s,\uptau)&=\bigg(\sinh\sqrt{\frac{s}{\uptau d}}y,\sqrt{\frac{s}{\uptau d}}\cosh\sqrt{\frac{s}{\uptau d}}y\bigg).
 	\end{split}
 \end{equation}
Obviously, we have  
 \begin{equation}\label{2.14}
 	\begin{split}
 		\frac{\partial \mf L}{\partial y}(y;s,\uptau)&=\sqrt{\frac{s}{\uptau d}} \mf C (y;s,\uptau),\\
 		\frac{\partial \mf C}{\partial y}(y;s,\uptau)&=\sqrt{\frac{s}{\uptau d}} \mf L(y;s,\uptau).
 	\end{split}
 \end{equation}
 Substituting \eqref{2.10} and \eqref{2.11} into \eqref{2.12} and using the vector notation, give the algebraic linear system 
 \begin{equation}\label{2.15} %\label{W}
 	\left( \begin{array}{cc}
 		\langle  \mf a ,\mf L(\alpha;s,\uptau)\rangle& \langle  \mf a ,\mf C(\alpha;s,\uptau)\rangle \\ \langle  \mf b ,\mf L(\beta;s,\uptau)\rangle &\langle  \mf b ,\mf C(\beta;s,\uptau)\rangle \end{array}\right) \;\left( \begin{array}{cc}
 		A \\ B\end{array}\right) \; =\left( \begin{array}{cc}
 		\bb G(s,\uptau)\\ H^{+}(x;s,\uptau)\end{array}\right),
 \end{equation}
 where $ \langle \cdot \rangle $ is the usual dot product in $ \bb R^{2} $, and
 \begin{equation}\label{2.16}
 	\begin{split}
 		\bb G(s,\uptau)&=\bb N[\lambda(t,\uptau)\zeta(t)]\\ 
 		H^{+}(x;s,\uptau)&=H(s,\uptau)-\sqrt{\frac{\uptau  d}{s}}\int_{\alpha}^{\beta}F(y;s,\uptau)\langle \mf b,\mf C(\beta-y;s,\uptau)\rangle dy
 	\end{split}
 \end{equation}
 with, $ H(s,\uptau)=\bb N[\lambda(t,\uptau)\xi(t)] $ and $\zeta(t)$ and $\xi(t)$ are the bounday data given in \eqref{2.3} and \eqref{2.4}, respectively.
 The solution ($A,B$) of system \eqref{2.15} is
 \begin{eqnarray} %\label{T}
 	%\begin{split}
 		&& A \Delta(s)=\bb G(s,\uptau)\langle  \mf b ,\mf C(\beta;s,\uptau)\rangle - H^{+}(x;s,\uptau)\langle  \mf a ,\mf C(\alpha;s,\uptau)\rangle ,\\ \label{2.17}
 		&& B \Delta(s)=-\left[\bb G(s,\uptau)\langle  \mf b ,\mf L(\beta;s,\uptau)\rangle - H^{+}(x;s,\uptau)\langle  \mf a ,\mf L(\alpha;s,\uptau)\rangle \right],\label{2.18}
 	%\end{split}
 \end{eqnarray}
 where
 \begin{equation}\label{2.19} %\label{pp}
 	\Delta(s) =\langle  \mf b ,\mf C(\beta;s,\uptau)\rangle \langle  \mf a ,\mf L(\alpha;s,\uptau)\rangle-\langle  \mf a ,\mf C(\alpha;s,\uptau)\rangle\langle  \mf b ,\mf L(\beta;s,\uptau)\rangle ,
 \end{equation}
is the determinant of the coefficient matrix of system \eqref{2.15}.
Substituting the constants $ A $ and $ B $ into \eqref{2.10} gives 
 \begin{equation*}%\label{2.20}
 	\begin{split}
 		\hat{\varphi}(x;s,\uptau)=&\frac{\cosh\sqrt{\frac{s}{\uptau d}}x}{\Delta(s)}\bigg(\bb G(s,\uptau)\langle  \mf b ,\mf C(\beta;s,\uptau)\rangle-H^{+}(x;s,\uptau)\langle  \mf a ,\mf C(\alpha;s,\uptau)\rangle\bigg)\\
 		&+\frac{\sinh\sqrt{\frac{s}{\uptau d}}x}{\Delta(s)} \bigg(-\bb G(s,\uptau)\langle  \mf b ,\mf L(\beta;s,\uptau)\rangle+H^{+}(x;s,\uptau)\langle  \mf a ,\mf L(\alpha;s,\uptau)\rangle\bigg)\\
 		&+\sqrt{\frac{\uptau  d}{s}}\int_{\alpha}^{x}\sinh\sqrt{\frac{s}{\uptau d}}(x-y)F(y;s,\uptau)dy,
 	\end{split}
 \end{equation*}
which can be rewritten as  %Hence eq.\eqref{KL} can be written as 
 \begin{equation}\label{2.20} %\label{DF}
 	\begin{split}
 		\hat{\varphi}(x;s,\uptau)=&\frac{\bb G(s,\uptau)\psi(x,\beta;s,\uptau,\mf b)}{\Delta(s)}-\frac{H(s,\uptau)\psi(x,\alpha;s,\uptau,\mf a)}{\Delta(s)}+\hat{\theta}(x;s,\uptau)
 	\end{split}
 \end{equation}
 where 
 \begin{equation}\label{2.21} %\label{SS}
 	\psi(x,y;s,\uptau,\textbf{L})=\langle  \textbf{L} ,\mf C(y;s,\uptau)\rangle\cosh\sqrt{\frac{s}{\uptau d}}x-\langle  \textbf{L} ,\mf L(y;s,\uptau)\rangle\sinh\sqrt{\frac{s}{\uptau d}}x,
 \end{equation}
 and
 \begin{equation*}%\label{2.22} %\label{FF}
 	\begin{split}
 		\hat{\theta}(x;s,\uptau)=&\frac{\psi(x,\alpha;s,\uptau,\mf a)}{\Delta(s)}\sqrt{\frac{\uptau  d}{s}}\int_{\alpha}^{\beta}\langle \mf b,\mf C(\beta-y;s,\uptau)\rangle F(y;s,\uptau)dy\\
 		&+\sqrt{\frac{\uptau  d}{s}}\int_{\alpha}^{x}\sinh\sqrt{\frac{s}{\uptau d}}(x-y)F(y;s,\uptau)dy.
 	\end{split}
 \end{equation*}
 For further computation we rewrite $\hat{\theta}(x;s,\uptau)$ as
\begin{align}
\hat{\theta}(x;s,\uptau)=&\sqrt{\frac{\uptau d}{s}}\bigg(\int_{\alpha}^{x}\frac{\psi(x,\alpha;s,\uptau,\mf a)\langle \mf b,\mf C(\beta-y;s,\uptau)\rangle+\Delta(s)\sinh\sqrt{\frac{s}{\uptau d}}(x-y)}{\Delta(s)} F(y;s,\uptau)dy\bigg) \nonumber\\
&+\sqrt{\frac{\uptau d}{s}}\int_{x}^{\beta} \frac{\psi(x,\alpha;s,\uptau,\mf a) \langle \mf b,\mf C(\beta-y;s,\uptau)\rangle}{\Delta(s)}  F(y;s,\uptau)dy.\label{2.22}
\end{align}
% 
%   
%%%%%%%%%%%%%%%%%%%%%%%%%%%%%%%%%%%%%%%%%%%%%%%%%%%%%%%%%%%%%%%
%%%%%%%%%%%%%%%%%%%%%%%%%%%%%%%%%%%%%%%%%%%%%%%%%%%%%%%%%%%%%%%%
\begin{lemma}\label{lem2.1}
 	Let $ \uptau, s, x,y \in \bb R$ and $ \textbf{L} \in \bb R^{2}$. Then 
\begin{equation}\label{2.23}
 		\psi(x,y;s,\uptau,\textbf{L})=\langle  \textbf{L} ,\mf C(y-x;s,\uptau)\rangle,
\end{equation}
%%%%%%%%%%%%%%%%%%%%%%%%%%%%%%%%%%%%%%
and
%%%%%%%
\begin{equation}\label{2.24}
\begin{split}
\Delta(s)\sinh\sqrt{\frac{s}{\uptau d}}(x-y)=&-\psi(x,\alpha;s,\uptau,\mf a)\langle \mf b ,\mf C(\beta-y;s,\uptau) \rangle\\
&+\langle \mf a, \mf C(\alpha-y;s,\uptau) \rangle \psi(x,\beta;s,\uptau,\mf b),
\end{split}
\end{equation}
consequently, for each zero $s$ of the function $\Delta(s)\sinh\sqrt{\frac{s}{\uptau d}}(x-y)$, one has
%%%%%%%
\begin{equation}\label{2.25}
\psi(x,\alpha;s,\uptau,\mf a)\langle \mf b ,\mf C(\beta-y;s,\uptau) \rangle
=\langle \mf a, \mf C(\alpha-y;s,\uptau) \rangle \psi(x,\beta;s,\uptau,\mf b),
\end{equation}
%%%% 
where $ \psi $ is given by \eqref{2.21}.
\end{lemma}
 \begin{proof}
The first two conclusions of the lemma follow directly from the uniqueness theorem of initial value problem for second order ordinary differential equations having constant coefficients. \\For fixed $ \uptau,s,y  \in \bb R$ and $ L \in \bb R^{2}$, in view of \eqref{2.13} and \eqref{2.21} the functions $\langle  \textbf{L} ,\mf C(y-x;s,\uptau)\rangle$ and $\psi(x,y;s,\uptau,\textbf{L})$ are solutions to the following initial value problem
%%%%%
\begin{equation}\label{2.26} %\label{TY}
 		\begin{split}
 			\frac{d^{2}z(x)}{dx^{2}}&-\frac{s}{\uptau d}z(x)=0\\
 			z(0)=&\langle  L ,\mf C(y;s,\uptau)\rangle,\quad z'(0)=-\sqrt{\frac{s}{\uptau d}}\langle  L ,\mf L(y;s,\uptau)\rangle.
 		\end{split}
\end{equation}
%%%%%%%%%%%%%%%%%
%	
%
Thus, with the uniqueness of the solution to Problem \eqref{2.26}, we conclude \eqref{2.23}.\\
It is easy to see that as functions in $x$ both sides of \eqref{2.24} solve the differential equation in \eqref{2.26} and satisfy the initial conditions
$$z(0)= - \Delta(s) \sinh\sqrt{\frac{s}{\uptau d}}y\ \mbox{  and } z'(0)=\sqrt{\frac{s}{\uptau d}}\cosh\sqrt{\frac{s}{\uptau d}}y.$$
Hence, by the uniqueness theorem \eqref{2.24} holds true.\\
For each $s_{*}$ being a zero of the function $\Delta(s)\sinh\sqrt{\frac{s}{\uptau d}}(x-y)$, taking the limit in both sides of \eqref{2.24} as $s\to s_{*}$, gives \eqref{2.25}. 
 \end{proof}	
Applying Lemma \ref{lem2.1}, \eqref{2.20} and \eqref{2.22} respectively can be reduced to % Now by applying Lemma \ref{Lam1}, eq.\eqref{DF} can be written as 
 \begin{equation}\label{2.27}%\label{PP}
 	\hat{\varphi}(x;s,\uptau)=\frac{\bb G(s,\uptau)
 		\langle \mf b  ,\mf C(\beta-x;s,\uptau)\rangle}{\Delta(s)}-\frac{H(s,\uptau)\langle \mf a  ,\mf C(\alpha-x;s,\uptau)\rangle}{\Delta(s)}+\hat{\theta}(x;s,\uptau),
 \end{equation}
 \begin{equation}\label{2.28}%\label{JJ}
 	\begin{split}
 		\hat{\theta}(x;s,\uptau)=&\sqrt{\frac{\uptau d}{s}}\int_{\alpha}^{x}\frac{\langle  \mf a ,\mf C(\alpha-y;s,\uptau)\rangle \langle  \mf b ,\mf C(\beta-x;s,\uptau)\rangle}{\Delta(s)} F(y;s,\uptau)dy\\
 		&+\sqrt{\frac{\uptau d}{s}}\int_{x}^{\beta}\frac{\langle  \mf a ,\mf C(\alpha-x;s,\uptau)\rangle \langle  \mf b ,\mf C(\beta-y;s,\uptau)\rangle}{\Delta(s)} F(y;s,\uptau)dy.
 	\end{split}
 \end{equation}
%%%%%%%%%%%%%%%%%%%%%%%%%%%%%%%%%%%%%%%%%%%%%%%%%%%%%%%%%%%%%%%%%%%%%%%%%%%%%%%%%%%%%%%%%%%%%%%%%%%%%%%%%%%%%%%%%%%%%%%%%%
Next, in order to obtain the solution to the initial value problem \eqref{2.1}-\eqref{2.4} we apply the inversion formula \eqref{N-Inversion} to \eqref{2.27} and \eqref{2.28}.
% In doing so, we assume that there exists a non zero sequence $ \{s_{k}\}_{k=1}^{\infty} $ of simple zeros of $\Delta (s)$. That is,
In doing so, we suppose that there are non zero simple roots $ \{s_{k}\}_{k=1}^{\infty} $ of  $\Delta (s)$. That is,
 \begin{equation}\label{2.29}  %\label{TTY}
	\Delta(s_{k})=0, \quad \Delta^{'}(s_{k})\neq0,\quad k=1,2,3,\cdots.
\end{equation} 
%%%%%%%%%%%%%%%%%
 \begin{lemma}\label{lem2.2} %\label{KKK}
 	Suppose that \eqref{2.29} holds true. For each $ x,y\in \bb R $ and $ t,\uptau>0$, we get
 	\begin{equation}\label{2.30}
 		\begin{split}
 			\bb N^{-1}\bigg\{\frac{\langle  \mf a ,\mf C(x;s,\uptau)\rangle \langle  \mf b ,\mf C(y;s,\uptau)\rangle}{\sqrt{\frac{s}{\uptau d}}\Delta(s)}\bigg\}=\Theta(x,y,t,\uptau)
 		\end{split}
 	\end{equation}
 	where,
 	\begin{equation}\label{2.31} %\label{EEEE}
 		\begin{split}
 			\Theta(x,y,t,\uptau)&=\Theta_{0}(x,y,\uptau)+\sum_{k=1}^{\infty}\frac{e^{\frac{s_{k}t}{\uptau}}\langle  \mf a ,\mf C(x;s_{k},\uptau)\rangle \langle  \mf b ,\mf C(y;s_{k},\uptau)\rangle}{\sqrt{\frac{s_{k}}{\uptau d}}\Delta^{'}(s_{k})},
 		\end{split}
 	\end{equation} 
 	with
 	\begin{equation}\label{2.32}   %\label{RRRR}
 	\Theta_{0}(x,y)=
 	\begin{cases}
 	0 & \mbox{if} \quad \frac{\langle  \mf a ,\mf C(x;s,\uptau)\rangle \langle  \mf b ,\mf C(y;s,\uptau)\rangle}{\sqrt{\frac{s}{\uptau d}}\Delta(s)}=O(1),\\
 	\lim\limits_{s\rightarrow 0}\sqrt{s \uptau d}\ \frac{\langle  \mf a ,\mf C(x;s,\uptau)\rangle \langle  \mf b ,\mf C(y;s,\uptau)\rangle}{\Delta(s)}  &   \mbox{if} \quad \frac{\langle  \mf a ,\mf C(x;s,\uptau)\rangle \langle  \mf b ,\mf C(y;s,\uptau)\rangle}{\sqrt{\frac{s}{\uptau d}}\Delta(s)}=O(\frac{1}{s}).
 	\end{cases}
 	\end{equation}%{1.1}
%%%%%%%%%%%%%%%%%%%%%%%%%%%%%%%%%%%%%%%%%%%%%%%%%%%%%%%%%%%%% 	
\begin{proof}
Let
\begin{equation}\label{2.33}
\mf R(s,\uptau)=\frac{\langle  \mf a ,\mf C(x;s,\uptau)\rangle \langle  \mf b ,\mf C(y;s,\uptau)\rangle}{\sqrt{\frac{s}{\uptau d}}\Delta(s)}.
\end{equation}
Applying the inversion formula \eqref{N-Inversion}, we find
%From \cite[Theorem 5.3]{BS}, the inverse natural transform of $\mf \Phi(s,\uptau)$ is 
 		\begin{equation}\label{2.34}
 	\Theta(x,y,t)= \bb N^{-1}\bigg\{ \mf R(s,\uptau)\bigg\} =\frac{1}{2\pi i} \int_{c-i\infty}^{c+i\infty} e^{\frac{st}{\uptau}}\mf R(s,\uptau)ds,\quad c,\uptau>0.
 		\end{equation}
The last integral can be usually calculated by the residue theorem \cite{lang2013complex}. Hence,  
\begin{equation}\label{2.35}
 			\Theta(x,y,t)=\sum_{\mbox{poles}\,\, s_{k}\mbox{ of }\mf R(s,\uptau)}\mbox{Res}[e^{\frac{st}{\uptau}}\mf R(s,\uptau);s_{k}].
\end{equation}
Recalling \eqref{2.29}, each $ s_{k}\,\,(k=1,2,\cdots) $ is a simple pole of $ e^{\frac{st}{\uptau}}\mf R(s,\uptau) $. Therefore,
 \begin{equation}\label{2.36}
 		 				\mbox{Res}[e^{\frac{st}{\uptau}}\mf R(s,\uptau);s_{k}]=\frac{e^{\frac{s_{k}t}{\uptau}}\langle  \mf a ,\mf C(x;s_{k},\uptau)\rangle \langle  \mf b ,\mf C(y;s_{k},\uptau)\rangle}{\sqrt{\frac{s_{k}}{\uptau d}}\Delta^{'}(s_{k})},\quad k=1,2,\cdots.
	\end{equation}
At $ s=0 $ we have
 		\begin{align}\label{2.37}
 				\mbox{Res}[e^{\frac{st}{\uptau}}\mf R(s,\uptau);0]&=\lim\limits_{s\rightarrow 0}s\ e^{\frac{st}{\uptau}}\mf R(s,\uptau)\nonumber\\
 				&=\lim\limits_{s\rightarrow 0}\frac{s}{\Delta(s)}\frac{e^{\frac{st}{\uptau}}\langle  \mf a ,\mf C(x;s,\uptau)\rangle \langle  \mf b ,\mf C(y;s,\uptau)\rangle}{\sqrt{\frac{s}{\uptau d}}}\nonumber\\
 				&=\lim\limits_{s\rightarrow 0}\sqrt{s \uptau d}\ \frac{\langle  \mf a ,\mf C(x;s,\uptau)\rangle \langle  \mf b ,\mf C(y;s,\uptau)\rangle}{\Delta(s)}.
 		\end{align}
 We see that either 
 		\begin{equation}\label{2.38}
 			\frac{\langle  \mf a ,\mf C(x;s,\uptau)\rangle \langle  \mf b ,\mf C(y;s,\uptau)\rangle}{\sqrt{\frac{s}{\uptau d}}\Delta(s)}=O(1) \quad \mbox{or} \quad \frac{\langle  \mf a ,\mf C(x;s,\uptau)\rangle \langle  \mf b ,\mf C(y;s,\uptau)\rangle}{\sqrt{\frac{s}{\uptau d}}\Delta(s)}=O(\frac{1}{s}),
 		\end{equation}
as $ s$ tends to $ 0 $. Then, $ s=0 $ is either a removable singular point or a simple pole of $   e^{\frac{st}{\uptau}}\mf R(s,\uptau)  $.\\
Hence, substituting \eqref{2.36} and \eqref{2.37} in \eqref{2.35} gives the main conclusion of the Lemma i.e., \eqref{2.31} and \eqref{2.32}.
%		Hence, using \eqref{2.36} and \eqref{2.37}, we can rewrite  \eqref{2.35} as follows
%		 		\begin{equation}\label{2.39}
%		% 			
%		 				\Theta(x,y,t)%=\sum_{\mbox{all poles of }\mf \Phi(s,\uptau)}Res[e^{\frac{st}{\uptau}}\mf \Phi(s,\uptau)]\\
%		 				=\Theta_{0}(x,y)+\sum_{k=1}^{\infty}\frac{e^{\frac{s_{k}t}{\uptau}}\langle  \mf a ,\mf C(x;s_{k},\uptau)\rangle \langle  \mf b ,\mf C(y;s_{k},\uptau)\rangle}{\sqrt{\frac{s_{k}}{\uptau d}}\Delta^{'}(s_{k})},
%		% 		
%		 		\end{equation} 
%		 		where,
%		 		\begin{equation}\label{2.40}     %\label{PDE}
%		 			\Theta_{0}(x,y)=
%		 			\begin{cases}
%		 				0\quad  \quad \quad \quad \quad \quad \quad\quad \quad\quad \quad\quad \quad\mbox{if} \quad \frac{\langle  \mf a ,\mf C(x;s,\uptau)\rangle \langle  \mf b ,\mf C(y;s,\uptau)\rangle}{\sqrt{\frac{s}{\uptau d}}\Delta(s)}=O(1)\\
%		 				\lim\limits_{s\rightarrow 0}\sqrt{s \uptau d}\ \frac{\langle  \mf a ,\mf C(x;s,\uptau)\rangle \langle  \mf b ,\mf C(y;s,\uptau)\rangle}{\Delta(s)}\quad \mbox{if} \quad \frac{\langle  \mf a ,\mf C(x;s,\uptau)\rangle \langle  \mf b ,\mf C(y;s,\uptau)\rangle}{\sqrt{\frac{s}{\uptau d}}\Delta(s)}=O(\frac{1}{s}).
%		 			\end{cases}
%		 		\end{equation}
 	\end{proof}
 \end{lemma}
In view of \eqref{2.30} of Lemma \ref{lem2.2} and \eqref{F}, \eqref{2.28} can be rewritten as 
 \begin{equation}\label{2.41}
 	\begin{split}
 		\hat{\theta}(x;s,\uptau)
 		=&-\frac{1}{d}\int_{\alpha}^{x}\bb N[\Theta(\alpha-y,\beta-x,t)]\, \bb N[\frac{r(y,t)}{(\frac{t^{m}}{\uptau^{m}}+\uptau^{m})^{\rho}}]dy\\
 		&-\frac{1}{\uptau d}\int_{\alpha}^{x}\bb N[\Theta(\alpha-y,\beta-x,t)]\eta(y)dy\\
 		&-\frac{1}{d}\int_{x}^{\beta}\bb N[\Theta(\alpha-x,\beta-y,t)]\, \bb N[\frac{r(y,t)}{(\frac{t^{m}}{\uptau^{m}}+\uptau^{m})^{\rho}}]dy\\
 		&-\frac{1}{\uptau d}\int_{x}^{\beta}\bb N[\Theta(\alpha-x,\beta-y,t)]\eta(y)dy.
 	\end{split}
 \end{equation}
 By the convolution formula \eqref{N-Convolution}, the inverse natural transform of \eqref{2.41} is
 \begin{equation}\label{2.42}%\label{DDDDD}
 	\begin{split}
 		\theta(x,t,\uptau)=&-\frac{1}{\uptau d}\int_{\alpha}^{x}\int_{0}^{t}\Theta(\alpha-y,\beta-x,t-\varsigma) \frac{r(y,\varsigma)}{(\frac{\varsigma ^m}{\uptau^{m}}+\uptau^{m})^{\rho}}d\varsigma dy\\
 		&-\frac{1}{\uptau d}\int_{\alpha}^{x}\Theta(\alpha-y,\beta-x,t)\eta(y)dy\\
 		&-\frac{1}{\uptau d}\int_{x}^{\beta}\int_{0}^{t}\Theta(\alpha-x,\beta-y,t-\varsigma)\frac{r(y,\varsigma)}{(\frac{\varsigma ^m}{\uptau^{m}}+\uptau^{m})^{\rho}}d\varsigma dy\\
 		&-\frac{1}{\uptau d}\int_{x}^{\beta}\Theta(\alpha-x,\beta-y,t)\eta(y)dy.
 	\end{split}
 \end{equation}
That is,
 \begin{equation}\label{2.43}  %\label{AH}
 	\begin{split}
 		\theta(x,t,\uptau)=&-\frac{1}{\uptau d}\int_{\alpha}^{x}\Theta_{0}(\alpha-y,\beta-x)\eta(y)dy-\frac{1}{\uptau d}\int_{x}^{\beta} \Theta_{0}(\alpha-x,\beta-y)\eta(y)dy\\
 		&-\frac{1}{\uptau d}\int_{\alpha}^{x}\int_{0}^{t}\Theta_0(\alpha-y,\beta-x) \frac{r(y,\varsigma)}{(\frac{\varsigma ^m}{\uptau^{m}}+\uptau^{m})^{\rho}}d\varsigma dy\\
 		&-\frac{1}{\uptau d}\int_{x}^{\beta}\int_{0}^{t}\Theta_{0}(\alpha-x,\beta-y)\frac{r(y,\varsigma)}{(\frac{\varsigma ^m}{\uptau^{m}}+\uptau^{m})^{\rho}}d\varsigma dy\\
 		&-\frac{1}{\uptau d}\sum_{k=1}^{\infty}\frac{\langle  \mf b ,\mf C(\beta-x;s_{k},\uptau)\rangle}{\sqrt{\frac{s_{k}}{\uptau d}}\Delta^{'}(s_{k})}\int_{\alpha}^{x}\int_{0}^{t} \frac{e^{s_{k}(t-\varsigma)/\uptau} r(y,\varsigma)\langle  \mf a ,\mf C(\alpha-y;s_{k},\uptau)\rangle}{(\frac{\varsigma ^m}{\uptau^{m}}+\uptau^{m})^{\rho}}d\varsigma dy\\ 
 		&-\frac{1}{\uptau d}\sum_{k=1}^{\infty}\frac{\langle  \mf a ,\mf C(\alpha-x;s_{k},\uptau)\rangle }{\sqrt{\frac{s_{k}}{\uptau d}}\Delta^{'}(s_{k})}\int_{x}^{\beta}\int_{0}^{t} \frac{e^{s_{k}(t-\varsigma)/\uptau} r(y,\varsigma)\langle  \mf b ,\mf C(\beta-y;s_{k},\uptau)\rangle}{(\frac{\varsigma ^m}{\uptau^{m}}+\uptau^{m})^{\rho}}d\varsigma dy\\
 		&-\frac{1}{\uptau d}\sum_{k=1}^{\infty}\frac{e^{\frac{s_{k}t}{\uptau}}\langle  \mf b ,\mf C(\beta-x;s_{k},\uptau)\rangle}{\sqrt{\frac{s_{k}}{\uptau d}}\Delta^{'}(s_{k})}\int_{\alpha}^{x}\langle  \mf a ,\mf C(\alpha-y;s_{k},\uptau)\rangle \eta(y)dy\\
 		&-\frac{1}{\uptau d}\sum_{k=1}^{\infty}\frac{e^{\frac{s_{k}t}{\uptau}}\langle  \mf a ,\mf C(\alpha-x;s_{k},\uptau)\rangle }{\sqrt{\frac{s_{k}}{\uptau d}}\Delta^{'}(s_{k})}\int_{x}^{\beta}\langle  \mf b ,\mf C(\beta-y;s_{k},\uptau)\rangle \eta(y)dy.
 	\end{split}
 \end{equation}
From Lemma \ref{lem2.1}, one has
 \begin{equation*}
 	\langle\mf a,\mf C(\alpha-x;s,\uptau)\rangle \langle\mf b,\mf C(\beta-y;s,\uptau)= \langle\mf a,\mf C(\alpha-y;s,\uptau)\rangle \langle\mf b,\mf C(\beta-x;s,\uptau)\rangle, 
 \end{equation*}
at $s=0$ and $s=s_{k}\, (k=1,2,\cdots)$ the zeros of $ \Delta(s)\sinh\sqrt{\frac{s}{\uptau d}}(x-y) $. That results in
$$ \Theta_{0}(\alpha-y,\beta-x)=\Theta_{0}(\alpha-x,\beta-y)$$
%\todo[color=green]{Here, $\varphi$ was replaced with $\vartheta$}
 \begin{equation*}
 	\begin{split}
 		\langle\mf a,\mf C(\alpha-x;s_{k},\uptau)\rangle & \int_{x}^{\beta} \langle\mf b,\mf C(\beta-y;s_{k},\uptau)\rangle \vartheta(y) dy\\
 		= & \langle\mf b,\mf C(\beta-x;s_{k},\uptau)\rangle \int_{x}^{\beta} \langle\mf a,\mf C(\alpha-y;s_{k},\uptau)\rangle \vartheta(y) dy.
 	\end{split}
 \end{equation*}
The first conclusion is obvious when $ \Theta_{0}=0 $ in  \eqref{2.32}.
Thus, \eqref{2.43} can be simplified as
 \begin{equation}\label{v}   %\label{K}
 	\begin{split}
 		\theta(x,t,\uptau)=&-\frac{1}{\uptau d}\int_{\alpha}^{\beta}\Theta_{0}(\alpha-x,\beta-y)\eta(y)dy\\
 		&-\frac{1}{\uptau d}\int_{\alpha}^{\beta}\int_{0}^{t}\Theta_0(\alpha-x,\beta-y) \frac{r(y,\varsigma)}{(\frac{\varsigma ^m}{\uptau^{m}}+\uptau^{m})^{\rho}}d\varsigma dy\\
 		&-\frac{1}{\uptau d}\sum_{k=1}^{\infty}\frac{\langle  \mf b ,\mf C(\beta-x;s_{k},\uptau)\rangle }{\sqrt{\frac{s_{k}}{\uptau d}}\Delta^{'}(s_{k})}\int_{\alpha}^{\beta}\int_{0}^{t}\frac{e^{s_{k}(t-\varsigma)/\uptau} r(y,\varsigma)\langle  \mf a ,\mf C(\alpha-y;s_{k},\uptau)\rangle}{(\frac{\varsigma ^m}{\uptau^{m}}+\uptau^{m})^{\rho}}d\varsigma dy\\ 
 		&-\frac{1}{\uptau d}\sum_{k=1}^{\infty}\frac{e^{s_{k}t/\uptau}\langle  \mf b ,\mf C(\beta-x;s_{k},\uptau)\rangle }{\sqrt{\frac{s_{k}}{\uptau d}}\Delta^{'}(s_{k})}\int_{\alpha}^{\beta}\langle  \mf a ,\mf C(\alpha-y;s_{k},\uptau)\rangle \eta(y)dy.
 	\end{split}
 \end{equation}
 %%%%%%%%%%%%%%%%%%%%%%%%%%%%%%%%%%%%%%%%%%%%%%%%%%%%%%%%%%%%%%%%%%%%%%%
 Next, we return to \eqref{2.27}. Using \eqref{N for derivatives} (for $n=1$) and \eqref{2.16}, \eqref{2.27} can be rewritten as  
 \begin{align}\label{2.45}    %\label{EEE}
 	\hat{\varphi}(x;s,\uptau) =& s\bb N[\lambda(t,\uptau)\zeta(t)]\frac{
 			\langle \mf b  ,\mf C(\beta-x;s,\uptau)\rangle}{s\Delta(s)}-s\bb N[\lambda(t,\uptau)\xi(t)]\frac{\langle \mf a  ,\mf C(\alpha-x;s,\uptau)\rangle}{s\Delta(s)} +\hat{\theta}(x;s,\uptau)\nonumber\\
 =& \big(\uptau \bb N[\frac{d}{dt}(\lambda(t,\uptau)\zeta(t))]+\lambda(0,\uptau)\zeta(0)\big)\frac{
 	\langle \mf b  ,\mf C(\beta-x;s,\uptau)\rangle}{s\Delta(s)} \nonumber\\
 &-\big(\uptau \bb N[\frac{d}{dt}(\lambda(t,\uptau)\xi(t))]+\lambda(0,\uptau)\xi(0)\big)\frac{\langle \mf a  ,\mf C(\alpha-x;s,\uptau)\rangle}{s\Delta(s)}+\hat{\theta}(x;s,\uptau).		
 \end{align}
 where $ \hat{\theta}(x;s,\uptau)$ is given in \eqref{2.28}. Now, we can obtain the solution $\varphi(x,t)$ of Problem \eqref{2.1}-\eqref{2.4} by operating the inversion formula \eqref{N-Inversion} in \eqref{2.45}. In doing so, we need the following lemma.
 \begin{lemma}\label{lem2.3}    % \label{Lam5}
 	Assume that \eqref{2.29} holds true. Then, for each $y \in \bb R,\,\  t,\uptau>0$ and $ \textbf{L} \in \bb R^{2}$, we get %$ \uprho$ 
 	\begin{equation}
 		\begin{split}
 			\bb N^{-1}\bigg\{\frac{\langle  \textbf{L} ,\mf C(y;s,\uptau)\rangle }{s\Delta(s)}\bigg\} =\Phi(y,t;\textbf{L})
 		\end{split}
 	\end{equation}
 	where,
 	\begin{equation}\label{R(y,t;l)}   %\label{E}
 		\begin{split}
 			\Phi(y,t;\textbf{L})&=\Phi_{0}(y;\textbf{L})+\sum_{k=1}^{\infty}\frac{e^{\frac{s_{k}t}{\uptau}} \langle  \textbf{L} ,\mf C(y;s_{k},\uptau)\rangle}{s_{k}\Delta^{'}(s_{k})},
 		\end{split}
 	\end{equation} %This result can be rewritten as
 	and
 	\begin{equation}\label{R_0}
 		\Phi_0(y;\textbf{L})=
 		\begin{cases}
 			\lim\limits_{s\rightarrow 0}\frac{ \langle  \textbf{L} ,\mf C(y;s,\uptau)\rangle}{\Delta(s)}\quad  \quad \quad  \quad  \mbox{if} \quad \frac{ \langle  \textbf{L} ,\mf C(y;s,\uptau)\rangle}{s\Delta(s)}=O(\frac{1}{s})\\
 			\lim\limits_{s\rightarrow 0}\frac{\partial}{\partial s}\bigg(\frac{s\langle  \textbf{L} ,\mf C(y;s,\uptau)\rangle}{\Delta(s)}\bigg)\quad \mbox{if} \quad \frac{ \langle  \textbf{L} ,\mf C(y;s,\uptau)\rangle}{s\Delta(s)}=O(\frac{1}{s^{2}}).
 		\end{cases}
 	\end{equation}
 \end{lemma}
 \begin{proof}The proof is similar to Lemma \ref{lem2.2}.
 \end{proof}
 From Lemma \ref{lem2.3} we see that
 \begin{align*}
 	\bb N^{-1}\bigg\{\frac{\langle\mf b,\mf C(\beta-x;s,\uptau)\rangle}{s\Delta(s)}\bigg\} =\Phi(\beta-x,t;\mf b),
 \quad \quad
	\bb N^{-1}\bigg\{\frac{\langle\mf a,\mf C(\alpha-x;s,\uptau)\rangle}{s\Delta(s)}\bigg\} =\Phi(\alpha-x,t;\mf a).
\end{align*}
 Hence, in view  of the convolution formula \eqref{N-Convolution} and the inversion of natural transform \eqref{N-Inversion}, inverting \eqref{2.45} yields
 \begin{align}
  		\varphi(x,t,\uptau)=&\int_{0}^{t}\Phi(\beta-x,t-\varsigma;\mf b)\big( \tilde{\zeta}^{'}(\varsigma)+\tilde{\zeta}(0)\delta_{0}(\varsigma)\big)d\varsigma\nonumber\\
 		&-\int_{0}^{t}\Phi(\alpha-x,t-\varsigma;\mf a)\big( \tilde{\xi}^{'}(\varsigma)+\tilde{\xi}(0)\delta_{0}(\varsigma)\big)d\varsigma+\theta(x,t,\uptau),
  \end{align}
 where $ \tilde{\zeta}=\lambda(t,\uptau)\zeta(t) $, $ \tilde{\xi}=\lambda(t,\uptau)\xi(t) $, $ \delta_{0} $ is the well-known Dirac delta function, and $\theta(x,t,\uptau)$ is given by \eqref{v}. Then, using the basic property of the Dirac delta function, that is $\delta_0 (\varsigma)\Phi(\varsigma)=\Phi(0)$, results in
 \begin{align*}%\label{WW}
 	 		\varphi(x,t,\uptau)=&\int_{0}^{t}\Phi(\beta-x,t-\varsigma;\mf b) \tilde{\zeta}^{'}(\varsigma)d\varsigma+\tilde{\zeta}(0)\Phi(\beta-x,t;\mf b) \nonumber\\
 		&-\int_{0}^{t}\Phi(\alpha-x,t-\varsigma;\mf a) \tilde{\xi}^{'}(\varsigma)d\varsigma-\tilde{\xi}(0)\Phi(\alpha-x,t;\mf a)+\theta(x,t,\uptau).
 \end{align*}
Integrating by parts, gives 
 \begin{align*}%\label{u}
 	\varphi(x,t,\uptau)=& \lambda(t,\uptau) \zeta(t) \Phi(\beta-x,0;\mf b)-\int_{0}^{t}\lambda(\xi,\uptau)D_{2}\Phi(\beta-x,t-\varsigma;\mf b)\zeta(\varsigma)d\varsigma \nonumber\\
 		&-\lambda(t,\uptau)\Phi(\alpha-x,0;\mf a) \xi(t)+\int_{0}^{t}\lambda(\xi,\uptau)D_{2}\Phi(\alpha-x,t-\varsigma;\mf a) \xi(\varsigma)d\varsigma \nonumber\\
 		&+\theta(x,t,\uptau),
 \end{align*}
 where $ D_{2}=\frac{\partial}{\partial t}\Phi(x,t;\bf b) $. Substituting from \eqref{R(y,t;l)}, gives
 \begin{align*}
 		\varphi(x,t,\uptau)=& \lambda(t,\uptau) \zeta(t) \left(\Phi_0(\beta-x;\mf b)+\sum_{k=1}^{\infty}\frac{ \langle  \mf b ,\mf C(\beta-x;s_{k},\uptau)\rangle}{s_{k} \Delta{'}(s_{k})}\right)\\
 		&-\int_{0}^{t}\lambda(\xi,\uptau)\zeta(\varsigma)\left(\sum_{k=1}^{\infty}\frac{e^{\frac{s_{k}(t-\varsigma)} {\uptau}} \langle  \mf b ,\mf C(\beta-x;s_{k},\uptau)\rangle}{ \Delta^{'}(s_{k})}\right) d\varsigma\\
 		&-\lambda(t,\uptau)\xi(t)\left(\Phi_0(\alpha-x;\mf a)+\sum_{k=1}^{\infty}\frac{ \langle  \mf a ,\mf C(\alpha-x;s_{k},\uptau)\rangle}{s_{k} \Delta^{'}(s_{k})}\right) \\
 		&+\int_{0}^{t}\lambda(\xi,\uptau)\xi(\varsigma)\left(\sum_{k=1}^{\infty}\frac{e^{\frac{s_{k}(t-\varsigma)}{\uptau}} \langle  \mf a ,\mf C(\alpha-x;s_{k},\uptau)\rangle}{ \Delta^{'}(s_{k})}\right) d\varsigma\\
 		&+\theta(x,t,\uptau),
 \end{align*}
 with $\theta(x,t,\uptau)$ is given by \eqref{v}. This result can be rewritten as 
 \begin{align*}%\label{u}%\label{UU}
 		\varphi(x,t,\uptau)=&\lambda(t,\uptau)\zeta(t)\Phi_0(\beta-x;\mf b)\\
 		&+\sum_{k=1}^{\infty}\frac{\zeta(t) \lambda(t,\uptau)-\int_{0}^{t}\lambda(\xi,\uptau)s_{k}e^{s_{k}(t-\varsigma)/\uptau}\zeta(\varsigma)d\varsigma}{s_{k} \Delta^{'}(s_{k})}\langle  \mf b ,\mf C(\beta-x;s_{k},\uptau)\rangle\\
 		&- \lambda(t,\uptau) \xi(t)\Phi_0(\alpha-x;\mf a)\\
 		&-\sum_{k=1}^{\infty}\frac{\lambda(t,\uptau) \xi(t) -\int_{0}^{t}\lambda(\xi,\uptau)s_{k}e^{s_{k}(t-\varsigma)/\uptau}\xi(\varsigma)d\varsigma}{s_{k} \Delta^{'}(s_{k})}\langle  \mf a ,\mf C(\alpha-x;s_{k},\uptau)\rangle\\
 		&+\theta(x,t,\uptau),
 \end{align*}
or
\begin{align}\label{u}
 		\varphi(x,t,\uptau)= & \lambda(t,\uptau) \zeta(t) \Phi_0(\beta-x;\mf b)+\sum_{k=1}^{\infty}\frac{ \Gamma_{k}\zeta(t)}{s_{k} \Delta^{'}(s_{k})}\langle  \mf b ,\mf C(\beta-x;s_{k},\uptau)\rangle \nonumber\\
 		&- \lambda(t,\uptau) \xi(t) \Phi_0(\alpha-x;\mf a)-\sum_{k=1}^{\infty}\frac{ \Gamma_{k}\xi(t)}{s_{k} \Delta^{'}(s_{k})}\langle  \mf a ,\mf C(\alpha-x;s_{k},\uptau)\rangle \nonumber\\
 		&+\theta(x,t,\uptau),
 \end{align}
 where $ \Gamma_{k} $ is the operator defined as 
 \begin{equation}\label{U_k}
 	\Gamma_{k}\phi(t)=\lambda(t,\uptau) \phi(t) -s_{k}\int_{0}^{t}\lambda(\xi,\uptau)e^{\frac{s_{k}(t-\varsigma)}{\uptau}}\phi(\varsigma)d\varsigma.
 \end{equation}
The integral in \eqref{U_k} is the Laplacian convolution formula for $\lambda(t,\uptau)\phi(t)$ with $e^{\frac{s_{k}t}{\uptau}}$. As a result, \eqref{u}, together with \eqref{v} and \eqref{U_k}, expresses the solution of Problem \eqref{2.1}-\eqref{2.4}.%(1.1-1.4).
 \begin{remark}
When $\rho=0$ and $\ r_j=\nu_j=0,\mbox{ for all } j=1,\cdots,n,$ Problem \eqref{2.1}-\eqref{2.4} and its solution
\begin{align*}  %\label{u(tau=1)}
	\varphi(x,t,\uptau=1)= & \zeta(t) \Phi_0(\beta-x;\mf b)+\sum_{k=1}^{\infty}\frac{\tilde{\Gamma}_{k}\zeta(t)}{s_{k} \Delta^{'}(s_{k})}\langle  \mf b ,\mf C(\beta-x,s_{k})\rangle \nonumber\\
	&-  \xi(t) \Phi_0(\alpha-x;\mf a)-\sum_{k=1}^{\infty}\frac{\tilde{\Gamma}_{k}\xi(t)}{s_{k} \Delta^{'}(s_{k})}\langle  \mf a ,\mf C(\alpha-x,s_{k})\rangle %\nonumber\\
	+\theta(x,t),
\end{align*}
with  % $ \tilde{U}_{k} $ is the operator 
$\Phi_0(\alpha-x;\mf a)$, $\mf C(y,s_{k})=\mf C(y;s_{k},\uptau=1)$ defined as \eqref{R_0}, \eqref{2.13}, respectively, 
\begin{equation*}     %\label{U_k}
	\tilde{\Gamma}_{k}\phi(t)= \phi(t) -s_{k}\int_{0}^{t} e^{s_{k}(t-\varsigma)}\phi(\varsigma)d\varsigma,
\end{equation*}
and
\begin{align*}    %\label{v(tau=1)}   %\label{K}
\theta(x,t)=&-\frac{1}{ d}\int_{\alpha}^{\beta}\Theta_{0}(\alpha-x,\beta-y)\eta(y)dy\\
%&-\frac{1}{ d}\int_{\alpha}^{\beta}\int_{0}^{t}\Theta_0(\alpha-x,\beta-y) \frac{r(y,\varsigma)}{(\frac{\varsigma ^m}{\uptau^{m}}+\uptau^{m})^{\rho}}d\varsigma dy\\
%&-\frac{1}{1 d}\sum_{k=1}^{\infty}\frac{e^{s_{k}t}\langle  \mf b ,\mf C(\beta-x;s_{k},\uptau)\rangle }{\sqrt{\frac{s_{k}}{d}} \Delta^{'}(s_{k})}\int_{\alpha}^{\beta}\int_{0}^{t}\frac{r(y,\varsigma)\langle  \mf a ,\mf C(\alpha-y;s_{k},\uptau)\rangle}{(\frac{\varsigma ^m}{\uptau^{m}}+\uptau^{m})^{\rho}}d\varsigma dy\\ 
&- \sum_{k=1}^{\infty}\frac{e^{s_{k}t}\langle  \mf b ,\mf C(\beta-x;s_{k},1)\rangle }{\sqrt{d s_{k}} \Delta^{'}(s_{k})}\int_{\alpha}^{\beta}\langle  \mf a ,\mf C(\alpha-y;s_{k},1)\rangle \eta(y)dy,
\end{align*}
%
% $u(x,t,\uptau=1)$ 
are reduced to that in \cite[Section 3]{rodrigo2016solution}.  
\end{remark} %(in this case we set $\uptau=1$)
%%%%%%%%%%%%%%%%%%%%%%%%%%%%%%%%%%%%%%%%%%%%%%%%%%%%%%%%%%%%%%%
%%%%%%%%%%%%%%%%%%%%%%%%%%%%%%%%%%%%%%%%%%%%%%%%%%%%%%%%%%%%%%%%
\section{Multilayer nonhomogeneous diffusion system}\label{sec3}
Here, we are seeking the solution of our main problem defined in \eqref{gPDE}-\eqref{Inner Bc2}, which was converted into a sequence of initial boundary value problems \eqref{BPV in 1 layer}-\eqref{BPV in n layer}. For the convenient of the reader and in order to draw the full picture in an easy way, we start with solving the bilayer diffusion problem in the following subsection, then we move to the general case in subsection \ref{subsec_nlayer}.
 \subsection{\textbf{Solution of a two-layer problem}}
 For the two-layer problem, we have
% Here, we consider the following problem
 \begin{align} 
 	&\frac{\partial \varphi_{1}}{\partial t}=d_{1} \frac{\partial^{2} \varphi_{1}}{\partial x^{2}}+\lambda(t,\uptau)r_{1}(x ,t), \quad x \in (x_{0},x_{1}),\ t,\uptau>0, \label{twolayer1}\\
 	&\varphi_{1}(x,0)=\eta_{1}(x),\quad x \in [x_{0},x_{1}],\\
 	&\imath \varphi_{1}(x_{0},t)+\iota \frac{\partial \varphi_{1}}{\partial x}(x_{0},t)=\lambda(t,\uptau)\zeta_{1}(t), \quad t \ge 0,\ \uptau >0, \\
 	&\nu_{1} \varphi_1(x_{1},t)+\mu_{1} \frac{\partial \varphi_{1}}{\partial x}(x_{1},t)=\lambda(t,\uptau)\xi_{1}(t), \quad t \ge 0,\ \uptau>0,
 \end{align}
 and
 \begin{align}
 	&\frac{\partial \varphi_{2}}{\partial t}=d_{2} \frac{\partial^{2} \varphi_{2}}{\partial x^{2}}+\lambda(t,\uptau)r_{2}(x ,t), \quad x \in (x_{1},x_{2}), \ t,\uptau>0,\\
 	&\varphi_{2}(x,0)=\eta_{2}(x),\quad x \in [x_{1},x_{2}],\\
 	&\nu_{2}\varphi_{2}(x_{1},t)+\mu_{2} \frac{\partial \varphi_{2}}{\partial x}(x_{1},t)=\lambda(t,\uptau)\zeta_{2}(t), \quad t \ge 0,\ \uptau >0,\\
 	&\ell \varphi_{2}(x_{2},t)+\l \frac{\partial \varphi_{2}}{\partial x}(x_{2},t)=\lambda(t,\uptau)\xi_{2}(t), \quad t \ge 0,\ \uptau>0. \label{twolayern}
 \end{align}
Similar to what we denote in Section \ref{sec2}, we define the following vector notation $  \mf a_{1}=(\imath,\iota) $, $  \mf b_{1}=(\nu_{1},\mu_{1}) $, $ \mf a_{2}=(\nu_{2},\mu_{2}) $, $ \mf b_{2}=(\ell,\l) $, and
 \begin{equation}
 	\begin{split}
 		\mf L_{1}(y;s,\uptau)&=\bigg(\cosh\sqrt{\frac{s}{\uptau d_{1}}}y,\sqrt{\frac{s}{\uptau d_{1}}}\sinh\sqrt{\frac{s}{\uptau d_{1}}}y\bigg),\\
 		\mf L_{2}(y;s,\uptau)&=\bigg(\cosh\sqrt{\frac{s}{\uptau d_{2}}}y,\sqrt{\frac{s}{\uptau d_{2}}}\sinh\sqrt{\frac{s}{\uptau d_{2}}}y\bigg),\\
 		\mf C_{1}(y;s,\uptau)&=\bigg(\sinh\sqrt{\frac{s}{\uptau d_{1}}}y,\sqrt{\frac{s}{\uptau d_{1}}}\cosh\sqrt{\frac{s}{\uptau d_{1}}}y\bigg),\\
 		\mf C_{2}(y;s,\uptau)&=\bigg(\sinh\sqrt{\frac{s}{\uptau d_{2}}}y,\sqrt{\frac{s}{\uptau d_{2}}}\cosh\sqrt{\frac{s}{\uptau d_{2}}}y\bigg).
 	\end{split}
 \end{equation}
Also, analogues to \eqref{2.19}, define 
 \begin{equation}
 	\begin{split}
 		\Delta_{1}(s)&=\langle  \mf b_{1} ,\mf C_{1}(x_{1},s,\uptau)\rangle \langle  \mf a_{1} ,\mf L_{1}(x_{0},s,\uptau)\rangle-\langle  \mf a_{1} ,\mf C_{1}(x_{0},s,\uptau)\rangle\langle  \mf b_{1} ,\mf L_{1}(x_{1},s,\uptau)\rangle,\\
 		\Delta_{2}(s)&=\langle  \mf b_{2} ,\mf C_{2}(x_{2},s,\uptau)\rangle \langle  \mf a_{2} ,\mf L_{2}(x_{1},s,\uptau)\rangle-\langle  \mf a_{2} ,\mf C_{2}(x_{1},s,\uptau)\rangle\langle  \mf b_{2} ,\mf L_{2}(x_{2},s,\uptau)\rangle.
 	\end{split}
 \end{equation}
Further, similar to \eqref{2.29}, suppose that there are nonzero simple roots $ \{s_{k}^{(1)}\} _{k=1}^{\infty}$ and $ \{s_{k}^{(2)}\} _{k=1}^{\infty}$ of the functions $	\Delta_{1}(s)$ and $\Delta_{2}(s)$, respectively. That is, 
 \begin{equation}
 	\Delta_{1}(s_{k}^{(1)})=0, \quad \Delta_{1}^{'}(s_{k}^{(1)})\neq0,\quad \Delta_{2}(s_{k}^{(2)})=0, \quad \Delta_{2}^{'}(s_{k}^{(2)})\neq0 \quad (k=1,2,...).  
 \end{equation}
 Therefore, according to \eqref{v}, we obtain
\begin{align}\label{v_1}   %\label{WWW}
  		\theta_{1}(x,t,\uptau)=&-\frac{1}{\uptau d_{1}}\int_{x_{0}}^{x_{1}}\Theta^{(1)}_{0}(x_{0}-x,x_{1}-y)\eta_{1}(y)dy \nonumber\\
 		&-\frac{1}{\uptau d_{1}}\int_{x_{0}}^{x_{1}}\int_{0}^{t}\Theta^{(1)}_0(x_{0}-x,x_{1}-y) \frac{r_{1}(y,\varsigma)}{(\frac{\varsigma ^m}{\uptau^{m}}+\uptau^{m})^{\rho}}d\varsigma dy \nonumber\\
 		&-\frac{1}{\uptau d_{1}}\sum_{k=1}^{\infty}\frac{\langle  \mf b_{1} ,\mf C_{1}(x_{1}-x,s^{(1)}_{k},\uptau)\rangle }{\sqrt{\frac{s^{(1)}_{k}}{\uptau d_{1}}}\Delta_{1}^{'}(s^{(1)}_{k})}\int_{x_{0}}^{x_{1}}\int_{0}^{t}\frac{e^{s^{(1)}_{k}(t-\varsigma)/\uptau}r_{1}(y,\varsigma)\langle  \mf a_{1} ,\mf C_{1}(x_{0}-y,s^{(1)}_{k},\uptau)\rangle}{(\frac{\varsigma ^m}{\uptau^{m}}+\uptau^{m})^{\rho}}d\varsigma dy \nonumber\\ 
 		&-\frac{1}{\uptau d_{1}}\sum_{k=1}^{\infty}\frac{e^{s^{(1)}_{k}t/\uptau}\langle  \mf b_{1} ,\mf C_{1}(x_{1}-x,s^{(1)}_{k},\uptau)\rangle }{\sqrt{\frac{s^{(1)}_{k}}{\uptau d_{1}}}\Delta_{1}^{'}(s^{(1)}_{k})}\int_{x_{0}}^{x_{1}}\langle  \mf a_{1} ,\mf C_{1}(x_{0}-y,s^{(1)}_{k},\uptau)\rangle \eta_{1}(y)dy
\end{align}
\begin{align}\label{v_2}
 		\theta_{2}(x,t,\uptau)=&-\frac{1}{\uptau d_{2}}\int_{x_{1}}^{x_{2}}\Theta^{(2)}_{0}(x_{1}-x,x_{2}-y)\eta_{2}(y)dy \nonumber\\
 		&-\frac{1}{\uptau d_{2}}\int_{x_{1}}^{x_{2}}\int_{0}^{t}\Theta^{(2)}_0(x_{1}-x,x_{2}-y) \frac{r_{2}(y,\varsigma)}{(\frac{\varsigma ^m}{\uptau^{m}}+\uptau^{m})^{\rho}}d\varsigma dy \nonumber\\
 		&-\frac{1}{\uptau d_{2}}\sum_{k=1}^{\infty}\frac{\langle  \mf b_{2} ,\mf C_{2}(x_{2}-x,s^{(2)}_{k},\uptau)\rangle }{\sqrt{\frac{s^{(2)}_{k}}{\uptau d_{2}}}\Delta_{2}^{'}(s^{(2)}_{k})}\int_{x_{1}}^{x_{2}}\int_{0}^{t}\frac{e^{s^{(2)}_{k}(t-\varsigma)/\uptau}r_{2}(y,\varsigma)\langle  \mf a_{2} ,\mf C_{2}(x_{1}-y,s^{(2)}_{k},\uptau)\rangle}{(\frac{\varsigma ^m}{\uptau^{m}}+\uptau^{m})^{\rho}}d\varsigma dy \nonumber\\ 
 		&-\frac{1}{\uptau d_{2}}\sum_{k=1}^{\infty}\frac{e^{s^{(2)}_{k}t/\uptau}\langle  \mf b_{2} ,\mf C_{2}(x_{2}-x,s^{(2)}_{k},\uptau)\rangle }{\sqrt{\frac{s^{(2)}_{k}}{\uptau d_{2}}}\Delta_{2}^{'}(s^{(2)}_{k})}\int_{x_{1}}^{x_{2}}\langle  \mf a_{2} ,\mf C_{2}(x_{1}-y,s^{(2)}_{k},\uptau)\rangle \eta_{2}(y)dy,
 \end{align}
 where, $ \Theta^{(1)}_0 $ and $ \Theta^{(2)}_0 $ can be defined as in Lemma \ref{lem2.2}.\\
 Also, similar to \eqref{u}, with the respective forms $ \Phi_0^{(1)} $ and $ \Phi_0^{(2)} $ from Lemma \ref{lem2.3} and the matching condition $ \xi_{1}(t) =\zeta_{2}(t) $ we get
 \begin{align}\label{u_1}   %\label{EE}
  		\varphi_{1}(x,t,\uptau)=& \lambda(t,\uptau)\zeta_{1}(t)\Phi^{(1)}_{0}(x_{1}-x;\mf b_{1})+\sum_{k=1}^{\infty}\frac{\Gamma^{(1)}_{k}\zeta_{1}(t)}{s^{(1)}_{k}\Delta_{1}^{'}(s^{(1)}_{k})}\langle  \mf b_{1} ,\mf C_{1}(x_{1}-x,s^{(1)}_{k},\uptau)\rangle \nonumber\\
 		&- \lambda(t,\uptau)\xi_{1}(t)\Phi^{(1)}_{0}(x_{0}-x;\mf a_{1})-\sum_{k=1}^{\infty}\frac{\Gamma^{(1)}_{k}\xi_{1}(t)}{s^{(1)}_{k}\Delta_{1}^{'}(s^{(1)}_{k})}\langle  \mf a_{1} ,\mf C_{1}(x_{0}-x,s^{(1)}_{k},\uptau)\rangle \nonumber\\
 		& + \theta_{1}(x,t,\uptau),
 \end{align}
 \begin{align}\label{u_2}
 		\varphi_{2}(x,t,\uptau)=&\lambda(t,\uptau) \xi_{1}(t)\Phi^{(2)}_{0}(x_{2}-x;\mf b_{2})+\sum_{k=1}^{\infty}\frac{\Gamma^{(2)}_{k}\xi_{1}(t)}{s^{(2)}_{k}\Delta_{2}^{'}(s^{(2)}_{k})}\langle  \mf b_{2} ,\mf C_{2}(x_{2}-x,s^{(2)}_{k},\uptau)\rangle \nonumber\\
 		&-  \lambda(t,\uptau)\xi_{2}(t)\Phi^{(2)}_{0}(x_{1}-x;\mf a_{2})-\sum_{k=1}^{\infty}\frac{\Gamma^{(2)}_{k}\xi_{2}(t)}{s^{(2)}_{k}\Delta_{2}^{'}(s^{(2)}_{k})}\langle  \mf a_{2} ,\mf C_{2}(x_{1}-x,s^{(2)}_{k},\uptau)\rangle \nonumber\\
 		&+ \theta_{2}(x,t,\uptau),
 \end{align}
 where the operators $ \Gamma_{k}^{(1)} $ and $  \Gamma_{k}^{(2)} $ are obtained from \eqref{U_k}. The matching condition $ \varphi_{1}(x_{1},t)=\Lambda_{1}\varphi_{2}(x_{1},t) $ yields
 \begin{align}\label{u_1 =k_1 u_2}   %\label{III}
  		\bigg( & \Lambda_{1} \lambda(t,\uptau)\Phi^{(2)}_{0}(x_{2}-x_{1};\mf b_{2})+\lambda(t,\uptau)\Phi^{(1)}_{0}(x_{0}-x_{1};\mf a_{1})\bigg)\xi_{1}(t) \nonumber\\
 		&+\sum_{k=1}^{\infty}\left(\frac{\langle  \mf a_{1} ,\mf C_{1}(x_{0}-x_{1},s^{(1)}_{k},\uptau)\rangle}{s^{(1)}_{k} \Delta_{1}^{'}(s^{(1)}_{k})} \Gamma^{(1)}_{k}\xi_{1}(t)-\frac{\Lambda_{1}\langle  \mf b_{2} ,\mf C_{2}(x_{2}-x_{1},s^{(2)}_{k},\uptau)\rangle}{s^{(2)}_{k} \Delta_{2}^{'}(s^{(2)}_{k})} \Gamma^{(2)}_{k}\xi_{1}(t)\right) \nonumber\\
 		= & \lambda(t,\uptau) \zeta_{1}(t)\Phi^{(1)}_{0}(0;\mf b_{1})+ \theta_{1}(x_{1},t,\uptau) 
 		+\sum_{k=1}^{\infty}\frac{\langle  \mf b_{1} ,\mf C_{1}(0,s^{(1)}_{k},\uptau)\rangle} {s^{(1)}_{k} \Delta_{1}^{'}(s^{(1)}_{k})} \Gamma^{(1)}_{k}\zeta_{1}(t) \nonumber\\
 		& + \Lambda_{1} \lambda(t,\uptau) \xi_{2}(t)\Phi^{(2)}_{0}(0;\mf a_{2})+\sum_{k=1}^{\infty}\frac{\Lambda_{1}\langle  \mf a_{2} ,\mf C_{2}(0,s^{(2)}_{k},\uptau)\rangle}{s^{(2)}_{k} \Delta_{2}^{'}(s^{(2)}_{k})} \Gamma^{(2)}_{k}\xi_{2}(t) \nonumber\\
 		& - \Lambda_{1} \theta_{2}(x_{1},t,\uptau).
 \end{align}
 For the unknown function $ \xi_{1} $ we can rewrite the linear integral equation \eqref{u_1 =k_1 u_2} as
 \begin{equation}\label{R}
 	\lambda(t,\uptau)\xi_{1}(t)+\sum_{k=1}^{\infty}\big(a_{k} \Gamma^{(1)}_{k}\xi_{1}(t)+b_{k} \Gamma^{(2)}_{k}\xi_{1}(t)\big)=c(t),
 \end{equation}
 where,
 \begin{equation}
 	\begin{split}
 		a_{k}&=\frac{\langle  \mf a_{1} ,\mf C_{1}(x_{0}-x_{1},s^{(1)}_{k},\uptau)\rangle}{s^{(1)}_{k} \Delta_{1}^{'}(s^{(1)}_{k})\big(\Lambda_{1} \Phi^{(2)}_{0}(x_{2}-x_{1};\mf b_{2})+ \Phi^{(1)}_{0}(x_{0}-x_{1};\mf a_{1})\big)},\\
 		b_{k}&=-\frac{\Lambda_{1}\langle  \mf b_{2} ,\mf C_{2}(x_{2}-x_{1},s^{(2)}_{k},\uptau)\rangle}{s^{(2)}_{k} \Delta_{2}^{'}(s^{(2)}_{k})\big(\Lambda_{1} \Phi^{(2)}_{0}(x_{2}-x_{1};\mf b_{2})+ \Phi^{(1)}_{0}(x_{0}-x_{1};\mf a_{1})\big)},
 	\end{split}
 \end{equation}
 and
 \begin{equation}
 	\begin{split}
 		c(t)=&\frac{1}{\Lambda_{1} \Phi^{(2)}_{0}(x_{2}-x_{1};\mf b_{2})+ \Phi^{(1)}_{0}(x_{0}-x_{1};\mf a_{1})}\bigg(\lambda(t,\uptau)\zeta_{1}(t)\Phi^{(1)}_{0}(0;\mf b_{1})\\
 		&+ \theta_{1}(x_{1},t,\uptau)+\sum_{k=1}^{\infty}\frac{\langle  \mf b_{1} ,\mf C_{1}(0,s^{(1)}_{k},\uptau)\rangle}{s^{(1)}_{k} \Delta_{1}^{'}(s^{(1)}_{k})} \Gamma^{(1)}_{k}\zeta_{1}(t)\\
 		&+\Lambda_{1} \lambda(t,\uptau)\xi_{2}(t)\Phi^{(2)}_{0}(0;\mf a_{2})+\sum_{k=1}^{\infty}\frac{\Lambda_{1}\langle  \mf a_{2} ,\mf C_{2}(0,s^{(2)}_{k},\uptau)\rangle}{s^{(2)}_{k} \Delta_{2}^{'}(s^{(2)}_{k})} \Gamma^{(2)}_{k}\xi_{2}(t)\\
 		&-\Lambda_{1} \theta_{2}(x_{1},t,\uptau)\bigg).
 	\end{split}
 \end{equation}
 Inspire of the convolution formula \eqref{N-Convolution}, the natural transform of \eqref{R} is
 \begin{equation*}     %\label{TT}
 	\begin{split}
 		\bb N[\xi_{1}(t)\lambda(t,\uptau)] &(s,\uptau)+\sum_{k=1}^{\infty}\bigg(a_{k}\big(\bb N[\xi_{1}(t)\lambda(t,\uptau)](s,\uptau)-\frac{\uptau s^{(1)}_{k}}{s-s^{(1)}_{k}}\bb N[\xi_{1}(t)\lambda(t,\uptau)]\big)\\
 		&+b_{k}\big(\bb N[\xi_{1}(t)\lambda(t,\uptau)](s,\uptau)-\frac{\uptau s^{(2)}_{k}}{s-s^{(2)}_{k}}\bb N[\xi_{1}(t)\lambda(t,\uptau)]\big)\bigg)=\bb N[c(t)](s,\uptau),
 	\end{split}
 \end{equation*}
 which can be rewritten as 
 \begin{equation*}
 	\begin{split}
 		\big(1-\bb N[\psi(t)](s,\uptau)\big)\bb N[\xi_{1}&(t)\lambda(t,\uptau)](s,\uptau)=\bb N[c(t)](s,\uptau).
 	\end{split}
 \end{equation*}
That is,
 \begin{equation}
	\begin{split}
		\bb N[\xi_{1}(t)\lambda(t,\uptau)](s,\uptau)=&\frac{1}{\big(1-\bb N[\psi(t)](s,\uptau)\big)}\bigg\{\bb N[c(t)](s,\uptau)\bigg\}\\
		=&\bigg(1+\bb N[\psi(t)](s,\uptau)+\big(\bb N[\psi(t)](s,\uptau)\big)^{2}+...\bigg)\bigg\{\bb N[c(t)](s,\uptau)\bigg\}\\
		=&\bb N[c(t)](s,\uptau)+\bb N[c(t)](s,\uptau)\sum_{m=1}^{\infty}\big(\bb N[\psi(t)](s,\uptau)\big)^{m}.
	\end{split}
\end{equation}
 where
 \begin{equation}\label{YYY}
 	\begin{split}
 		\bb N[\psi(t)](s,\uptau)=&-\sum_{k=1}^{\infty}\bigg(a_{k}\big(1-\frac{\uptau s^{(1)}_{k}}{s-s^{(1)}_{k}}\big)+b_{k}\big(1-\frac{\uptau s^{(2)}_{k}}{s-s^{(2)}_{k}}\big)\bigg),
 	\end{split}
 \end{equation}
for which the inverse natural transform is
 \begin{equation}
 	\begin{split}
 		\psi(t)=&-\uptau\sum_{k=1}^{\infty}\big(a_{k}+b_{k}\big)\delta_{0}(t)
 		+\uptau\sum_{k=1}^{\infty}\bigg(a_{k}s^{(1)}_{k}e^{\frac{s_{k}^{(1)}}{\uptau}t}+b_{k}s^{(2)}_{k}e^{\frac{s_{k}^{(2)}}{\uptau}t}\bigg),
 	\end{split}
 \end{equation}
 where $ \delta_{0} $ is the Dirac delta function. Hence, we have 
%
% %
 \begin{equation}\label{xi_1}    %\label{Q}
	\begin{split}
		\xi_{1}(t)=&\frac{1}{\lambda(t,\uptau)}\left(c(t)+\sum_{m=1}^{\infty}\psi_{m}\ast c(t)\right),
	\end{split}
\end{equation}
 where for $ m\geq 2 $, $ \psi_{m} $ is the $ m- $times self-convolution of $ \psi $. %That is, \[\psi_{m}=\underbrace{\psi*\cdots*\psi}_{m- {\text{times }}}.\]
Thus, one can conclude the solution to the bilayer diffusion problem \eqref{twolayer1}-\eqref{twolayern} by the formulas \eqref{u_1} and \eqref{u_2}, together with \eqref{v_1} and \eqref{v_2}, with $ \xi_{1} $ be given in \eqref{xi_1}. Now, it's time to attack the main problem in the following subsection.
%%%%%%%%%%%%%%%%%%%%%%%%%%%%%%%%%%%%%%%%%%%%%%%%%%%%%%%%%%%%%%%%
 \subsection{\textbf{Solution of a multi-layer problem}}\label{subsec_nlayer}
 Here, we investigate the solvability of the main problem \eqref{gPDE}-\eqref{Inner Bc2}, through solving the initial boundary value problems \eqref{BPV in 1 layer}-\eqref{BPV in n layer}.
 Similar to what we have denoted in Section 2, we consider the following notations
 \begin{equation}
	\mf a_{j}=
	\begin{cases}
		(\imath,\iota),\quad \quad\quad\quad\quad\ j=1,\\
		(\nu_j,\mu_j),\quad \qquad \  \quad 2\leq j\le n,
	\end{cases}
\end{equation}
 \begin{equation}
	\mf b_{j}=
	\begin{cases}
		(\nu_j,\mu_j),\quad \qquad \  1\le j\le n-1,\\
		(\ell,\l),\quad \quad\quad\quad\quad\ \quad j=n,
	\end{cases}
\end{equation}
 and, , for all $ 1\le j\le n $
 \begin{equation}
 	\begin{split}
 		\mf L_{j}(y;s,\uptau)&=\bigg(\cosh\sqrt{\frac{s}{\uptau d_{j}}}y,\sqrt{\frac{s}{\uptau d_{j}}}\sinh\sqrt{\frac{s}{\uptau d_{j}}}y\bigg),\\
 		\mf C_{j}(y;s,\uptau)&=\bigg(\sinh\sqrt{\frac{s}{\uptau d_{j}}}y,\sqrt{\frac{s}{\uptau d_{j}}}\cosh\sqrt{\frac{s}{\uptau d_{j}}}y\bigg).\\
 	\end{split}
 \end{equation}
 %%%%%%%%%%%
 Moreover, define
 \begin{equation}
 	\begin{split}
 		\Delta_{j}(s)&=\langle  \mf b_{j} ,\mf C_{j}(x_{j},s,\uptau)\rangle \langle  \mf a_{j} ,\mf L_{j}(x_{j-1},s,\uptau)\rangle-\langle  \mf a_{j} ,\mf C_{j}(x_{j-1},s,\uptau)\rangle\langle  \mf b_{j} ,\mf L_{j}(x_{j},s,\uptau)\rangle,
 	\end{split}
 \end{equation}
 and let $ \{s_{k}^{(j)}\} _{k=1}^{\infty}$ be the sequence of zeros of the function $\Delta_{j}(s)$ for all $ 1\le j\le n $, i.e.,
 \begin{equation}
 	\Delta_{j}(s_{k}^{(j)})=0, \quad \Delta_{j}^{'}(s_{k}^{(j)})\neq0,\quad (k=1,2,...).  
 \end{equation}
 Analogue to the computations of \eqref{v} and \eqref{u}, we have for the current case, for all $j=1,\cdots,n$
 \begin{align}
 %	\begin{split}
 		\theta_{j}(x,t,\uptau)=&-\frac{1}{\uptau d_{j}}\int_{x_{j-1}}^{x_{j}}\Theta^{(j)}_{0}(x_{j-1}-x,x_{j}-y)\eta_{j}(y)dy\nonumber\\
 		&-\frac{1}{\uptau d_{j}}\int_{x_{j-1}}^{x_{j}}\int_{0}^{t}\Theta^{(j)}_0(x_{j-1}-x,x_{j}-y) \frac{r_{j}(y,\varsigma)}{(\frac{\varsigma ^m}{\uptau^{m}}+\uptau^{m})^{\rho}}d\varsigma dy\nonumber\\
 		&-\frac{1}{\uptau d_{j}}\sum_{k=1}^{\infty}\frac{\langle  \mf b_{j} ,\mf C_{j}(x_{j}-x,s_{k}^{(j)},\uptau)\rangle }{\sqrt{\frac{s_{k}^{(j)}}{\uptau d_j}}\Delta_j^{'}(s_{k}^{(j)})}\int_{x_{j-1}}^{x^{j}}\int_{0}^{t}\frac{e^{s_{k}^{(j)}(t-\varsigma)/\uptau}r_{j}(y,\varsigma)\langle  \mf a_{j} ,\mf C(x_{j-1}-y,s_{k}^{(j)},\uptau)\rangle}{(\frac{\varsigma ^m}{\uptau^{m}}+\uptau^{m})^{\rho}}d\varsigma dy\nonumber\\ 
 		&-\frac{1}{\uptau d_{j}}\sum_{k=1}^{\infty}\frac{e^{s_{k}^{(j)}t/\uptau}\langle  \mf b_{j} ,\mf C(x_{j}-x,s_{k}^{(j)},\uptau)\rangle }{\sqrt{\frac{s_{k}^{(j)}}{\uptau d_j}}\Delta_{j}^{'}(s_{k}^{(j)})}\int_{x_{j-1}}^{x_{j}}\langle  \mf a_{j} ,\mf C(x_{j-1}-y,s_{k}^{(j)},\uptau)\rangle \eta_{j}(y)dy,\label{vjnlayer}
 %	\end{split}
 \end{align}
 where, $ \Theta^{(j)}_0 $ can be defined in a similar way as in Lemma \ref{lem2.2}, and 
 \begin{align}
 	%\begin{split}
 		\varphi_{j}(x,t,\uptau)=&\lambda(t,\uptau) \zeta_{j}(t) \Phi_j(x_{j}-x,0;\mf b_{j})-\int_{0}^{t}\lambda(\xi,\uptau)D_{2}\Phi_j(x_{j}-x,t-\varsigma;\mf b_{j})\zeta_{j}(\varsigma)d\varsigma \nonumber\\
 		&-\lambda(t,\uptau) \xi_{j}(t) \Phi_j(x_{j-1}-x,0;\mf a_{j}) +\int_{0}^{t}\lambda(\xi,\uptau)D_{2}\Phi_j(x_{j-1}-x,t-\varsigma;\mf a_{j}) \xi_{j}(\varsigma)d\varsigma\nonumber\\
 		&+ \theta_{j}(x,t,\uptau),\label{RRR}
 %	\end{split}
 \end{align}
with the respective forms $ \Phi_j $ defined by \eqref{R(y,t;l)} in Lemma \ref{lem2.3}.
This last equation \eqref{RRR} can be rewritten as 
 \begin{equation}
 	\begin{split}
 		\varphi_{j}(x,t,\uptau)=T_{j}\tilde{\zeta}_{j}(x_{j}-x,t;\mf b_{j})-T_{j}\tilde{\xi}_{j}(x_{j-1}-x,t;\mf a_{j})+ \theta_{j}(x,t,\uptau),
 	\end{split}
 \end{equation}
 in which $\tilde{\zeta}_{j}=\lambda(t,\uptau) \zeta_{j}(t)$, $\tilde{\xi}_{j}=\lambda(t,\uptau) \xi_{j}(t)$ and the linear operator $ T_{j}$ is defined by
 \begin{equation}\label{RAGa}
 	T_{j}\varphi(y,t;\mb L)= \Phi_j(y,0;\mb L)\varphi(t)-\int_{0}^{t}D_{2}\Phi_j(y,t-\varsigma;\mb L)\varphi(\varsigma)d\varsigma
 \end{equation}
for all $ j=1,\cdots ,n,\ \mb L \in \bb{R}^2 $.\\
 The matching conditions $ \varphi_{j}(x_{j},t,\uptau)=\Lambda_{j}\varphi_{j+1}(x_{j},t,\uptau) $, $j=1,...,n-1$, lead to
 \begin{align}
 	%\begin{split}
 		T_{j}\tilde{\zeta}_{j}(0,t;\mf b_{j})&-T_j\tilde{\xi}_{j}(x_{j-1}-x_{j},t;\mf a_{j})+ \theta{j}(x_{j},t,\uptau)\nonumber\\
 		=&\Lambda_{j}\bigg(T_{j+1}\tilde{\zeta}_{j+1}(x_{j+1}-x_{j},t;\mf b_{j+1})-T_{j+1}\tilde{\xi}_{j+1}(0,t;\mf a_{j+1})+ \theta_{j+1}(x_{j},t,\uptau)\bigg).\label{YYYY}
 %	\end{split}
 \end{align}
 In the sprite of the matching conditions \eqref{g_j=h_j-1}, we have $ \tilde \zeta_{j+1}(t)=\tilde \xi_{j}(t)$ for all $1\le j\le n-1 $. Thus, for $j=1$, 
 \begin{equation}\label{A1}
 	\begin{split}
 		-T_{1}\tilde{\xi}_{1}(x_{0}-x_{1},t;\mf a_{1})&-\Lambda_{1}T_{2}\tilde{\xi}_{1}(x_{2}-x_{1},t;\mf b_{2})+\Lambda_{1}T_{2}\tilde{\xi}_{2}(0,t;\mf a_{2})\\
 		&=\Lambda_{1} \theta_{2}(x_{1},t,\uptau)- \theta_{1}(x_{1},t,\uptau)-T_{1}\tilde{\zeta}_{1}(0,t;\mf b_{1}).
 	\end{split}
 \end{equation}
For $2\le j\le n-2$
 \begin{equation}\label{A2}
 	\begin{split}
 		T_{j}\tilde{\xi}_{j-1}(0,t;\mf b_{j})&-T_{j}\tilde{\xi}_{j}(x_{j-1}-x_{j},t;\mf a_{j})-\Lambda_{j}T_{j+1}\tilde{\xi}_{j}(x_{j+1}-x_{j},t;\mf b_{j+1})\\
 		&+\Lambda_{j}T_{j+1}\tilde{\xi}_{j+1}(0,t;\mf a_{j+1})=\Lambda_{j} \theta_{j+1}(x_{j},t,\uptau)- \theta_{j}(x_{j},t,\uptau)
 	\end{split}
 \end{equation}
For $j=n-1$
 \begin{equation}\label{A3}
 	\begin{split}
 		T_{n-1}\tilde{\xi}_{n-2}(0,t;\mf b_{n-1})&-T_{n-1}\tilde{\xi}_{n-1}(x_{n-2}-x_{n-1},t;\mf a_{n-1})-\Lambda_{n-1}T_{n}\tilde{\xi}_{n-1}(x_{n}-x_{n-1},t;\mf b_{n})\\
 		&=\Lambda_{n-1} \theta_{n}(x_{n-1},t,\uptau)- \theta_{n-1}(x_{n-1},t,\uptau)-\Lambda_{n-1}T_{n}\tilde{\xi}_{n}(0,t;\mf a_{n}).
 	\end{split}
 \end{equation}
This system \eqref{A1},\eqref{A2} and \eqref{A3}, of $(n-1)$ integral equations of the unknowns $\tilde{\xi}_j ;\ 1\leq j\leq n-1$, can be adjusted as a matrix equation
 \begin{equation}\label{NN}
 \mathcal A(0)\textbf{h}(t)+(\mathcal A^{'}\ast \textbf{h})(t)= \textbf{b}(t),
 \end{equation}
 with $ \mathcal A(t) $ is a tridiagonal matrix of order $ n-1 $ whose entries:\\   %In the main diagonal
$$ 	-\Phi_j(x_{j-1}-x_{j},t;\mf a_{j})-\Lambda_{j}\Phi_j(x_{j+1}-x_{j},t;\mf b_{j+1}),\quad 1\le j\le n-1\quad ({\text{main diagonal}}),$$
 %The super-diagonal 
  $$ \Lambda_{j}\Phi_{j+1}(0,t;\mf a_{j+1}), \quad 1\le j\le n-2, \quad (\text{{super diagonal}})$$ 
%  and along the sub-diagonal
  $$ \Phi_j(0,t;\mf b_{j}),\quad  2\leq j\leq n-1 \quad ({\text{subdiagonal}}),$$
%\begin{align*} -\Phi_j(x_{j-1}-x_{j},t;\mf a_{j})-\Lambda_{j}\Phi_j(x_{j+1}-x_{j},t;\mf b_{j+1}),\quad 1\le j\le n-1,\,  ({\text{main diagonal}})\\
%  \Lambda_{j}\Phi_{j+1}(0,t;\mf a_{j+1}), \quad 1\le j\le n-2,\,  ({\text{auper diagonal}})\\
%   \Phi_j(0,t;\mf b_{j}),\quad  2\leq j\leq n-1,\,  ({\text{subdiagonal}})
%\end{align*}
% 
and the vectors $  \textbf{h}(t) $ and $ \textbf{b}(t) $ are defined as
 \begin{equation}
 	\textbf{h}(t)=
 	\begin{pmatrix}
 		\tilde{\xi}_{1}(t)\\
 		\vdots\\
 		\tilde{\xi}_{n-1}(t)
 	\end{pmatrix}
 \end{equation}
 and 
 \begin{equation}
 	\textbf{b}(t)=
 	\begin{pmatrix}
 		\Lambda_{1}\theta_{2}(x_{1},t,\uptau)- \theta_{1}(x_{1},t,\uptau)-T_{1}\tilde{\zeta}_{1}(0,t;\mf b_{1})\\
 		\Lambda_{2} \theta_{3}(x_{2},t,\uptau)-\theta_{2}(x_{2},t,\uptau)\\
 		\vdots\\
 		\Lambda_{n-2} \theta_{n-1}(x_{n-2},t,\uptau)- \theta_{n-2}(x_{n-2},t,\uptau)\\
 		\Lambda_{n-1} \theta_{n}(x_{n-1},t,\uptau)- \theta_{n-1}(x_{n-1},t,\uptau)-\Lambda_{n-1}T_{n}\tilde{\xi}_{n}(0,t;\mf a_{n})
 	\end{pmatrix}.
 \end{equation}
In fact, we can rewrite \eqref{NN} as 
 \begin{equation}\label{G}
 	\textbf{h}(t)=\mathcal C(t)+(\mathcal B\ast \textbf{h})(t),
 \end{equation} 
 with $ \mathcal C(t)=\mathcal A(0)^{-1} \textbf{b}(t),$ and $\mathcal B(t)=-\mathcal A(0)^{-1}\mathcal A^{'}(t)$. In view of the convolution formula \eqref{N-Convolution}, the natural transform of \eqref{G} reads
 \begin{equation}\label{S}
 	\bb N[\textbf{h}(t)](s,\uptau)=\bb N[\mathcal C(t)](s,\uptau)+\uptau\bb N[\mathcal B(t)](s,\uptau) \ \bb N[\textbf{h}(t)](s,\uptau),
 \end{equation} 
which is equivalent to, 
 \begin{equation}\label{D}
 	\begin{split}
 		\bb N[\textbf{h}(t)](s,\uptau)&=\bigg(I-\uptau\bb N[\mathcal B(t)](s,\uptau)\bigg)^{-1}\bb N[\mathcal C(t)](s,\uptau)\\
 		&=\bigg(I+\uptau\bb N[\mathcal B(t)](s,\uptau)+(\uptau\bb N[\mathcal B(t)](s,\uptau))^{2}+...\bigg)\bb N[\mathcal C(t)](s,\uptau),
 	\end{split}
 \end{equation} 
 where $ I $ is the $ (n-1)\times (n-1) $ identity matrix. Once again throughout the convolution sence \eqref{N-Convolution}, the natural transform inversion of \eqref{D} is
 \begin{equation}
 	\begin{split}
 		\textbf{h}(t)&=\mathcal C(t)+(\mathcal B\ast\mathcal C)(t)+(\mathcal B\ast\mathcal B\ast\mathcal C)(t)+...\\
 		&=\mathcal C(t)+\sum_{m=1}^{\infty}(\mathcal B_{m}\ast\mathcal C)(t),
 	\end{split}
 \end{equation}
where, $ \mathcal B_{m} $ is the $ m- $times self-convolution of $ \mathcal B $. Finally, the solution of the nonhomogeneous multilayer diffusion sytems \eqref{BPV in 1 layer}-\eqref{BPV in n layer} and hence that of the main problem \eqref{gPDE}-\eqref{Inner Bc2} is concluded as
 \begin{equation}
 	\begin{split}
 		\varphi_{j}(x,t,\uptau)=&\lambda(t,\uptau)\zeta_{j}(t)\Phi_0^{(j)}(x_{j}-x;\mf b_{j})-\lambda(t,\uptau)\xi_{j}(t)\Phi_0^{(j)}(x_{j-1}-x;\mf a_{j})\\
 		&+\sum_{k=1}^{\infty}\frac{\lambda(t,\uptau)\zeta_{j}(t) -\int_{0}^{t}\lambda(\xi,\uptau)s_{k}^{(j)}e^{s_{k}^{(j)}(t-\varsigma)/\uptau}\zeta_{j}(\varsigma)d\varsigma} {s_{k}^{(j)}\Delta_j^{'}(s_{k}^{(j)})}\langle  \mf b_{j} ,\mf C_{j}(x_{j}-x,s_{k}^{(j)},\uptau)\rangle\\
 		&-\sum_{k=1}^{\infty}\frac{\lambda(t,\uptau)\xi_{j}(t) -\int_{0}^{t}\lambda(\xi,\uptau)s_{k}^{(j)}e^{s_{k}^{(j)}(t-\varsigma)/\uptau}\xi_{j}(\varsigma)d\varsigma} {s_{k}^{(j)}\Delta_j^{'}(s_{k}^{(j)})}\langle  \mf a_{j} ,\mf C_{j}(x_{j-1}-x,s_{k}^{(j)},\uptau)\rangle\\
 		&+ \theta_{j}(x,t,\uptau),
 	\end{split}
 \end{equation}
with the respective forms $ \Phi_0^{(j)} $ and $\theta_j$ defined as in \eqref{R_0} and \eqref{vjnlayer}, respectively for all $ j=1,...,n $. 
%%%%%%%%%%%%%%%%%%%%%%%%%%%%%%%%%%%%%%%%%%%%%%%%%%%%%%%%%%%%%%%%%%%%%%%%%%%%%%%%%%%%%%%%%%%%%%%%%%%%%%%%%
\section{Conclusion}
Throughout the current contribution, a one-dimensional n-layer nonhomogeneous diffusion problem with time-varying data and general interface conditions have been concluded by means of a generalized integral transform. Although, most of the previous works have been focused on solving the problems of the homogeneous diffusion equation, the nonhomogeneous diffusion equation problem arises in many physical application. We have obtained the exact solutions for one- and multi-layer nonhomogeneous diffusion problems. The former case has been solved by a new generalized integral transform, the later one (n-layer problem) has been recast in a sequence of one layer problems. The obtained results generalize and extend those in \cite{carr2016semi},\cite{carr2018semi}, \cite{park2009one}, \cite{rodrigo2016solution} and \cite{zimmerman2016analytical}. 

Our results motivate to deal with other types of diffusion problems. For example, reaction diffusion problems, Advection-reaction diffusion problems and non-autonomous reaction diffusion problems, etc.

On the other hand, more general partial differential equations (PDEs) and systems can be considered. for example, system of coupled PDEs, nonlinear diffusion PDEs and  non-autonomous reaction diffusion PDEs. Those kinds of PDEs appear widely as epidemiological models to study and analyze the spread of diseases and pandemics \cite{du2018partial,hickson2009critical,raimundez2021covid,viguerie2021simulating}.
%%%%%%%%%%%%%%%%%%%%%%%%%%%%%%%%%%%%%%%%%%%%%%%%%%%%%%%%%%%%%%%%%% % 
\section*{Acknowledgment}
%The authors thank the editor and the referees for their valuable comments and suggestions which improved greatly the quality of this paper.
This project was supported by the Academy of  Scientific Research and Tecchnology (ASRT), Egypt (Grant N0. 6407)
%%%%%%%%%%%%%%%%%%%%%%%%%%%%%%%%%%%%%%%%%%%%%%%%%%%%%%%%%%%%%%%%%%%%%%%%%%%%%%%%%%%%%%%%%%%%%%%%%%%%%%%%%%%%%%%%%%%%%%%%%%%%%%%%%%%%%%%%%%%%%%%%%%%%%%%%%%%%%%%%%%%%%%%%%%%%%%%%%%%%%%%%%%%%%%%%%%%%%%%%%%%%%%%%%%%%%%%%%%%%%%%%%%%%
%\section*{References}
% 
\bibliography{References_Akel1.bib}
\end{document}